\documentclass[final,leqno]{siamltex}
\usepackage{amsmath}
\usepackage{graphicx}
\usepackage[notcite,notref]{showkeys}
\usepackage{mathrsfs}
\usepackage{graphics}
\usepackage{float}
\usepackage{amsfonts,amssymb}
\usepackage{dsfont}
\usepackage{pifont}
\usepackage{wrapfig} 
\usepackage{hyperref}
\usepackage{multirow}
\usepackage{color}
\numberwithin{equation}{section}
\usepackage{threeparttable}
\def\3bar{{|\hspace{-.02in}|\hspace{-.02in}|}}
\def\E{{\mathcal{E}}}
\def\T{{\mathcal{T}}}

\def\pT{{\partial T}}

\def\W{{\mathcal{W}}}

\def\bw{{\mathbf{w}}}
\def\bE{{\mathbf{E}}}
\def\bt{{\mathbf{t}}}
\def\bB{{\mathbf{B}}}
\def\bD{{\mathbf{D}}}
\def\bH{{\mathbf{H}}}
\def\bu{{\mathbf{u}}}
\def\bv{{\mathbf{v}}}
\def\bH{{\mathbf{H}}}
\def\bn{{\mathbf{n}}}

\def\bq{{\mathbf{q}}}
\def\bE{{\mathbf{E}}}
\def\be{{\mathbf{e}}}
\def\bw{{\mathbf{w}}}
\def\bf{{\mathbf{f}}}
\def\bQ{{\mathbf{Q}}}
\def\bj{{\mathbf{j}}}
\def\bphi{{\boldsymbol{\phi}}}
\def\bzeta{{\boldsymbol{\zeta}}}

\def\bphi{{\boldsymbol{\phi}}}
\def\ljump{{[\![}}
\def\rjump{{]\!]}}

\newtheorem{remark}{Remark}[section]
\newtheorem{algorithm}{Weak Galerkin Algorithm}[section]

\def\ad#1{\begin{aligned}#1\end{aligned}}  \def\b#1{\mathbf{#1}} 
 \def\an#1{\begin{align}#1\end{align}} 
 
\def\p#1{\begin{pmatrix}#1\end{pmatrix}}

\title{Efficient Weak Galerkin Finite Element Methods for Maxwell Equations on polyhedral Meshes without Convexity Constraints}

\begin{document}
 \author{
 Chunmei Wang \thanks{Department of Mathematics, University of Florida, Gainesville, FL 32611, USA (chunmei.wang@ufl.edu). The research of Chunmei Wang was partially supported by National Science Foundation Grant DMS-2136380.}
  \and
 Shangyou Zhang\thanks{Department of Mathematical Sciences,  University of Delaware, Newark, DE 19716, USA (szhang@udel.edu).  }
 }

\maketitle
\begin{abstract}
This paper presents an efficient weak Galerkin (WG) finite element method with reduced stabilizers for solving the time-harmonic Maxwell equations on both convex and non-convex polyhedral meshes. By employing bubble functions as a critical analytical tool, the proposed method enhances efficiency by partially eliminating the stabilizers traditionally used in WG methods.
This streamlined WG method demonstrates stability and effectiveness on convex and non-convex polyhedral meshes, representing a significant improvement over existing stabilizer-free WG methods, which are typically limited to convex elements within finite element partitions.
 The method achieves an optimal error estimate for the exact solution in a discrete 
$H^1$
 norm, and additionally, an optimal 
$L^2$
 error estimate is established for the WG solution. Several numerical experiments are conducted to validate the method's efficiency and accuracy.
\end{abstract}



\pagestyle{myheadings}

\section{Introduction}
In this paper, we aim to develop an efficient weak Galerkin finite element method with reduced stabilizers specifically designed for solving the Maxwell equations without convexity constraints. The Maxwell equations are coupled magnetic and electric equations described by:
\begin{equation}\label{model} 
\begin{split}
\nabla\times\bE=&-\frac{\partial\bB}{\partial t},\\
\nabla\times\bH=&\frac{\partial\bD}{\partial t}+\bj,\\
\nabla\cdot\bD=&\rho,\\
\nabla\cdot\bB=&0,
\end{split}
\end{equation}
along with the constitutive relations:
$$
\bB=\mu\bH,\qquad \bj=\sigma\bE, \qquad\bD=\epsilon\bE.
$$
Here $\bE$ represents the electric field intensity, $\bH$ is the magnetic field
intensity,  $\bD$ is the displacement current density, $\bB$ denotes  the magnetic flux density,    $\bj$  signifies the electric current density, $\mu$ is the
permeability, $\sigma$ is the electric conductivity, $\rho$ is the charge density,   and $\epsilon$ is the  permittivity.

For time-harmonic fields, where the time dependence is assumed to follow a harmonic form, i.e.,
$exp(i\omega t)$,  the Maxwell equations \eqref{model} can be re-expressed for the Fourier transform of the fields by utilizing the constitutive relations as follows
\begin{equation} \label{model2}
\begin{split}
\nabla\times\bE=&-i\omega\mu\bH,\\
\nabla\times\bH=&i\omega\epsilon\bE+\sigma\bE,\\
\nabla\cdot(\epsilon\bE)=&\rho,\\
\nabla\cdot(\mu\bH)=&0.
\end{split}
\end{equation}
By applying the operator $\nabla\times \mu^{-1}$ to the first equation in \eqref{model2}
 and incorporating the second equation in \eqref{model2}, we derive
 \begin{equation}\label{model3}
  \nabla\times(\mu^{-1}\nabla\times\bE)=(\omega^2 \epsilon-i \omega\sigma)\bE.   
 \end{equation}
Thus, the electric field intensity $\bE$  can be solved using   \eqref{model3} and the third equation in \eqref{model2}, along with appropriate boundary conditions. Subsequently, the magnetic field intensity  $\bH$ can be determined using the second and fourth equations in \eqref{model2}, together with the boundary conditions.

For simplicity, we focus on the equations for the electric field intensity $\bE$ with $\omega=0$ in
a heterogeneous medium $\Omega\subset\mathbb R^3$. A generalized mixed formulation of this model problem
aims to find unknown vector-valued function  $\bu$ and scalar-valued $p$ that satisfy
\begin{equation} \label{model4}
\begin{split}
\nabla\times (\mu^{-1}\nabla\times \bu)-\epsilon\nabla p=&\bf_1, \qquad \text{in}\quad\Omega, \\
\nabla\cdot(\epsilon \bu)=&g_1, \qquad \text{in}\quad\Omega, \\
 \bu \times\bn=&\bphi, \qquad \text{on}\quad\partial\Omega, \\
  p=&0, \qquad \text{on}\quad\partial\Omega.
  \end{split}
\end{equation}

Over the past several decades, extensive research has been conducted on the Maxwell equations by various scholars. The 
$H(curl)$ conforming finite element method was initially introduced in \cite{wang2} and later advanced in \cite{wang3}. Discontinuous Galerkin (DG) finite element methods have also been developed for the Maxwell equations \cite{wang4, wang5, wang6, wang7, wang8}, including a notable mixed DG formulation introduced and analyzed in \cite{wang6}. Recently, advancements include a weakly over-penalized symmetric interior penalty method, as introduced and analyzed in \cite{wang9}. Additionally, weak Galerkin finite element methods have been proposed in \cite{wangmaxwell, mwg, yemaxwell}. Numerous other numerical methods have also been developed to discretize the Maxwell equations.

The weak Galerkin   finite element method is a modern numerical technique for solving partial differential equations (PDEs), where differential operators are reconstructed within a distribution framework for piecewise polynomials. The regularity requirements are mitigated through carefully designed stabilizers. The WG method, demonstrated in \cite{wg1, wg2, wg3, wg4, wg5, wg6, wg7, wg8, wg9, wg10, wg11, wg12, wg13, wg14, wangmaxwell, wg16,
wg17, wg18, wg19, wg20, wg21, itera, wy3655}, and notably the primal-dual weak Galerkin (PDWG) approach  \cite{pdwg1, pdwg2, pdwg3, pdwg4, pdwg5, pdwg6, pdwg7, pdwg8, pdwg9, pdwg10, pdwg11, pdwg12, pdwg13, pdwg14, pdwg15}, uses weak derivatives and weak continuities based on conventional weak forms of PDEs, ensuring stability and accuracy. PDWG addresses challenges in traditional numerical techniques by treating numerical solutions as constrained minimizations of functionals, leading to a symmetric scheme involving both primal and dual variables.

This paper introduces an efficient  weak Galerkin  finite element method for solving Maxwell's equations. Our contributions are summarized as follows:
 \textbf{1. Reduction of Stabilizers:} Compared to traditional stabilizer-dependent WG methods, our method reduces the need for certain stabilizers. Specifically, by utilizing higher-degree polynomials to compute the discrete weak curl, one stabilizer for the discrete weak curl is eliminated, while the stabilizer for the discrete weak gradient is retained. The proposed method maintains the size and global sparsity of the stiffness matrix, thereby reducing programming complexity compared to traditional stabilizer-dependent WG methods.
\textbf{2. Applicability to Non-Convex polyhedral Meshes:} Unlike existing stabilizer-free WG methods \cite{ye1, ye2}, which typically work only on convex elements, our method is applicable to both convex and non-convex elements within polyhedral meshes. This is achieved by employing bubble functions as a critical analytical tool, allowing us to circumvent the restrictive conditions commonly imposed by other stabilizer-free WG methods. This flexibility facilitates broader generalization to various  PDEs without the usual implementation challenges. Recently, one of the authors has proposed stabilizer-free WG methods that do not require the convexity assumption, applied to solving the Poisson  equation and the biharmonic equation  \cite{wangst1, wangst2}.  To the best of our knowledge, this paper is the first to propose  WG methods with reduced stabilizers for the Maxwell equations. 
Theoretical analysis demonstrates that our method achieves an optimal error estimate for WG approximations in a discrete 
$H^1$ norm, and an optimal order 
$L^2$ error estimate is established for the WG solution. Numerical experiments are conducted to validate the theoretical findings, showcasing the method's effectiveness in practical applications.

The paper is organized as follows: Section 2 derives a weak formulation for the Maxwell equations \eqref{model}. Section 3 reviews the discrete weak gradient and discrete weak curl operators. Section 4 introduces the simplified weak Galerkin algorithm for the Maxwell equations. In Section 5, we establish the existence and uniqueness of the solutions for the WG scheme. Section 6 focuses on deriving the error equations for the WG scheme. Section 7 provides an optimal order error estimate in a discrete  $H^1$ norm  for the WG approximations. Section 8 presents an optimal 
$L^2$ error estimate for the WG solution under certain regularity assumptions. Finally, Section 9 demonstrates the numerical performance of the WG algorithm through various test examples.

We follow the standard notations for Sobolev spaces and norms defined on a given open and bounded domain  $D\subset \mathbb{R}^3$ with Lipschitz continuous boundary.  Let $\|\cdot\|_{s,D}$, $|\cdot|_{s,D}$ and $(\cdot,\cdot)_{s,D}$ denote the norm, seminorm, and inner product in the Sobolev space $H^s(D)$ for any $s\ge 0$. The space $H^0(D)$ coincides with $L^2(D)$, the space of square-integrable functions, for which the norm and inner product are denoted by  $\|\cdot \|_{D}$ and $(\cdot,\cdot)_{D}$	
 , respectively. When $D=\Omega$ or when the domain of integration is clear from the context, we shall omit the subscript $D$ in the norm and the inner product notation. The symbol 
  $C$ denotes a generic constant independent of the meshsize and other physical or functional parameters. 


 \section{Weak Formulations}\label{Section:2}
In this section, we aim to derive the weak formulation for the Maxwell equations \eqref{model}. To this end, we first introduce the following  spaces:
$$
H(curl; \Omega)=\{\bu\in [L^2(\Omega)]^3: \nabla\times\bu\in [L^2(\Omega)]^3\},
$$  
$$
H_0(curl; \Omega) =\{\bu\in H(curl; \Omega): \bn\times \bu=0\ \text{on}\  \partial\Omega\},
$$
$$
H(div; \Omega)=\{\bu\in [L^2(\Omega)]^3: \nabla\cdot\bu\in L^2(\Omega)\},
$$  
$$
H_0^1(\Omega)=\{q\in H^1(\Omega): q=0\ \text{on}\ \partial\Omega\}, 
$$
 where $\bn$ denotes the unit outward normal vector to the boundary $\partial\Omega$.

By applying the standard integration by parts, we can derive a weak formulation for the Maxwell model equations \eqref{model4}. This weak formulation seeks 
$(\bu, p)\in H(curl; \Omega)\times H_0^1(\Omega)$ such that $\bu\times \bn=\bphi$ on $\partial \Omega$ and 
\begin{equation}\label{weakform}
    \begin{split}
         (\nu\nabla\times\bu, \nabla\times\bv) -(\bv, \nabla p)=&(\bf, \bv), \qquad \forall\bv\in H_0(curl; \Omega),\\
      (\bu, \nabla q)=&-(g, q), \qquad \forall q\in H_0^1(\Omega),
    \end{split}
\end{equation}
where 
$\nu=\frac{1}{\mu \epsilon}$, $\bf=\frac{\bf_1}{\epsilon}$, and $g=\frac{g_1}{\epsilon}$.

\section{Weak Differential Operators}\label{Section:Hessian}
The principal differential operators in the weak formulation (\ref{weakform}) for the Maxwell equations  (\ref{model}) are the gradient operator $\nabla$ and the  curl operator $\nabla \times$. In this section, we  briefly review the discrete weak gradient operator \cite{pdwg4, pdwg10} and the discrete weak curl operator \cite{wangmaxwell, wg9}.

Let $T$ be a polyhedral domain with boundary $\partial T$. A scalar-valued weak function on $T$ refers to $\sigma=\{\sigma_0,\sigma_b\}$  where $\sigma_0\in L^2(T)$ and $\sigma_b\in L^{2}(\partial T)$. Here $\sigma_0$ and $\sigma_b$  represent the values of $\sigma$ in the interior and on the boundary of $T$, respectively. Note that $\sigma_b$  may not necessarily be the trace of $\sigma_0$   on $\partial T$. Denote by $\W(T)$ the space of scalar-valued weak functions on $T$, i.e.,
\begin{equation*}\label{2.1}
\W(T)=\{\sigma=\{\sigma_0,\sigma_b\}: \sigma_0\in L^2(T), \sigma_b\in
L^{2}(\partial T)\}.
\end{equation*}
A vector-valued weak function on $T$ refers to  $\bv=\{\bv_0,\bv_b\}$  where  $\bv_0\in [L^2(T)]^3$ and $\bv_b\in
[L^{2}(\partial T)]^3$. Here,  $\bv_0$ and $\bv_b$ represent the values of $\bv$ in the interior and on the boundary of $T$, respectively.  Note that $\bv_b$  may not necessarily be the trace of $\bv_0$  on $\partial T$. Denote by $V(T)$ the space of vector-valued weak functions on $T$, i.e.,
\begin{equation*}
V(T)=\{\bv=\{\bv_0,\bv_b  \}: \bv_0\in [L^2(T)]^3,  \bv_b\in
[L^{2}(\partial T)]^3\}.
\end{equation*}
 
The weak gradient of $\sigma\in \W(T)$, denoted by $\nabla_w \sigma$, is defined as a linear functional on $[H^1(T)]^3$ such that
\begin{equation*}
(\nabla_w  \sigma,\boldsymbol{\psi})_T=-(\sigma_0,\nabla \cdot \boldsymbol{\psi})_T+\langle \sigma_b,\boldsymbol{\psi}\cdot \textbf{n}\rangle_{\partial T},
\end{equation*}
for all $\boldsymbol{\psi}\in [H^1(T)]^3$.

The weak curl  operator of any $\bv\in V(T)$, denoted by $\nabla_w\times\bv$,  is defined in the dual space of $H(curl; T)$. Its action on $\bq\in H(curl; T)$ is given by 
$$
(\nabla _w\times \bv, \bq)_T=(\bv_0, \nabla\times \bq)_T-\langle \bv_b\times\bn,  \bq\rangle_{\partial T}.
$$

The weak divergence operator of any $\bv\in V(T)$, denoted by $\nabla_w\cdot\bv$,  is defined in the dual space of $H^1(T)$. Its action on $q\in H^1(T)$ is given by 
$$
(\nabla _w\cdot\bv, q)_T=-(\bv_0, \nabla q)_T+\langle \bv_b\cdot\bn,   q\rangle_{\partial T}.
$$

Denote by $P_r(T)$ the space of polynomials on $T$ with degree no more than $r$. 
A discrete version of $\nabla_{w}\sigma$  for $\sigma\in \W(T)$, denoted by $\nabla_{w, r, T}\sigma$, is defined as the unique polynomial vector in $[P_r(T) ]^3$ satisfying
\begin{equation}\label{disgradient}
(\nabla_{w, r, T} \sigma, \boldsymbol{\psi})_T=-(\sigma_0, \nabla \cdot \boldsymbol{\psi})_T+\langle \sigma_b, \boldsymbol{\psi} \cdot \textbf{n}\rangle_{\partial T}, 
\end{equation}
 for all $\boldsymbol{\psi}\in [P_r(T)]^3$. The usual integration by parts gives:
 \begin{equation}\label{disgradient*}
 (\nabla_{w, r, T} \sigma, \boldsymbol{\psi})_T= (\nabla \sigma_0, \boldsymbol{\psi})_T-\langle \sigma_0- \sigma_b, \boldsymbol{\psi} \cdot \textbf{n}\rangle_{\partial T},  
 \end{equation}
 for any $\boldsymbol{\psi}\in [P_r(T)]^3$, provided that $\sigma_0\in H^1(T)$.

A discrete version of $\nabla_w\times\bv$ for $\bv\in V(T)$, denoted by $\nabla_{w, r, T}\times \bv$, is defined as the unique polynomial vector in $[P_r(T) ]^3$ satisfying
\begin{equation}\label{discurlcurl}
(\nabla_{w, r, T} \times  \bv, \bq)_T=(\bv_0,  \nabla\times  \bq)_T-\langle \bv_b\times\bn,  \bq\rangle_{\partial T}, 
\end{equation}
 for any $\bq \in [P_r(T)]^3$. The usual integration by parts yields:
\begin{equation}\label{discurlcurlnew}
\begin{split}
 (\nabla_{w, r, T} \times \bv, \bq)_T = (  \nabla\times \bv_0,     \bq)_T-\langle (\bv_b-\bv_0)\times\bn,  \bq\rangle_{\partial T}, 
\end{split} 
\end{equation}
 for any $\bq \in [P_r(T)]^3$, provided that $\bv_0\in H(curl; T)$.

A discrete version of $\nabla_w\cdot\bv$ for $\bv\in V(T)$, denoted by $\nabla_{w, r, T}\cdot\bv$, is defined as the unique polynomial  in $P_r(T)$ satisfying
\begin{equation}\label{disdiv}
(\nabla_{w, r, T} \cdot  \bv, q)_T=-(\bv_0,  \nabla   q)_T+\langle \bv_b\cdot\bn,  q\rangle_{\partial T}, 
\end{equation}
 for any $q \in P_r(T)$. The usual integration by parts yields:
\begin{equation}\label{disdivnew}
\begin{split}
 (\nabla_{w, r, T} \cdot \bv, q)_T = (  \nabla\cdot \bv_0,   q)_T-\langle (\bv_0-\bv_b)\cdot\bn,   q\rangle_{\partial T}, 
\end{split} 
\end{equation}
 for any $q \in P_r(T)$, provided that $\bv_0\in H(div; T)$.
\section{Weak Galerkin Algorithm}\label{Section:WGFEM}
Let ${\cal T}_h$ be a finite element partition of the domain $\Omega\subset\mathbb R^3$ consisting of shape-regular polyhedrons, as defined in \cite{wy3655}. Denote by ${\mathcal E}_h$ the set of  all   faces  in ${\cal T}_h$ , and let ${\mathcal E}_h^0={\mathcal E}_h \setminus
\partial\Omega$ represent the set of all interior faces.  Let $h_T$ denote the mesh size of each element $T\in {\cal T}_h$, and let 
 $h=\max_{T\in {\cal T}_h}h_T$ be the mesh size for the partition ${\cal T}_h$.

Let $k\geq 1$. Denote by
$W_k(T)$ the local discrete space of the scalar-valued weak functions, given by:
\an{\label{Wk}
W_k(T)=\{\{\sigma_0,\sigma_b\}:\sigma_0\in P_{k-1}(T),\sigma_b\in
P_k(e),e\subset \partial T\}.
  }
Furthermore, let $\bt_1$ and $\bt_2$ be two tangential unit vectors on face $e\in \mathcal{E}_h$.
Denote by
$V_k(T)$ the local discrete space of the vector-valued weak functions, given by:
\an{\label{Vk}
V_k(T)=\{\{\bv_0,\bv_b=v_1\bt_1+v_2\bt_2\}:\bv_0\in [P_k(T)]^3,  v_1, v_2\in
P_k(e),  e\subset \partial T\}.  }
By patching  $W_k(T)$ over all the elements $T\in {\cal T}_h$
through a common value $\sigma_b$ on the interior interface $\E_h^0$, we obtain the scalar-valued weak finite element space $W_h$:
$$
W_h=\big\{\{\sigma_0, \sigma_b\}:\{\sigma_0, \sigma_b\}|_T\in W_k(T), \forall T\in {\cal T}_h \big\}.
$$
The subspace of $W_h$ with vanishing boundary values on $\partial\Omega$  is denoted by  $W_h^0$: 
\begin{equation*}\label{W0}
W_h^0=\{\{\sigma_0, \sigma_b\}\in W_h: \sigma_b=0\ \text{on}\ \partial \Omega\}.
\end{equation*}
Similarly, by patching $V_k(T)$ over all elements $T\in {\cal T}_h$
through a common value $\bv_b$ on the interior interface $\E_h^0$, we obtain the vector-valued weak finite element space $V_h$:
$$
V_h=\big\{\{\bv_0,\bv_b\}:\{\bv_0,\bv_b\}|_T\in V_k(T), \forall T\in {\cal T}_h \big\}.
$$
The subspace of  $V_h$ with vanishing boundary values on $\partial\Omega$ is denoted by 
 $V_h^0$:
\begin{equation*}\label{V0}
V_h^0=\big\{\{\bv_0,\bv_b\}\in V_h: \bv_b\times\bn=0\ \text{on}\ \partial\Omega\big\}.
\end{equation*}

Let $r$  be an integer. For simplicity of notation and to avoid confusion,  for any $\sigma\in
W_h$ and $\bv\in V_h$, denote by $\nabla_{w}\sigma$ and $\nabla_w \times  \bv$ the discrete weak actions   $\nabla_{w, k, T}\sigma$ and  $\nabla_{w, r, T}  \times\bv$ computed using   (\ref{disgradient}) and \eqref{discurlcurl} on each element $T$, respectively:
$$
(\nabla_{w}\sigma)|_T= \nabla_{w, k, T}(\sigma|_T), \qquad \sigma\in W_h,
$$
 $$
 (\nabla_w\times  \bv)|_T= \nabla_{w, r, T}\times (\bv|_T), \qquad \bv\in V_h.
$$

For any $p, q\in W_h$ and $\bu, \bv\in V_h$, we introduce the
following bilinear forms:
\begin{eqnarray*}\label{EQ:local-stabilizer}
a(\bu, \bv)=&\sum_{T\in {\cal T}_h}a_T(\bu, \bv),\\
b(\bu, q)=&\sum_{T\in {\cal T}_h}b_T(\bu, q), \\
s(p, q)=&\sum_{T\in {\cal T}_h}s_T(p, q), 
\label{EQ:local-bterm}
\end{eqnarray*} 
where
\begin{equation*}
\begin{split}
a_T(\bu, \bv) =& (\nu\nabla_w\times\bu, \nabla_w\times \bv)_T,\\
b_T(\bu,q)=&(\bu_0, \nabla_w q)_T,\\
s_T (p, q)=& h_T\langle p_0-p_b, q_0-q_b\rangle_{\partial T}.  
 \end{split}
\end{equation*} 
 
The following is the simplified weak Galerkin scheme, which eliminates part of the stabilizers, for the Maxwell equations \eqref{model} based on the variational formulation \eqref{weakform}:
\begin{algorithm}\label{a-1}
Find $(\bu_h; p_h)\in V_h  \times W_{h}^0$, satisfying $\bu_b\times \bn=Q_b\bphi$ on $\partial\Omega$ and
\begin{eqnarray}\label{32}
 a(\bu_h, \bv)-b(\bv, p_h)&=& (\bf, \bv_0), \qquad \forall \bv\in V_{h}^0,\\
s(p_h, q) +b(\bu_h, q)&=&-(g, q_0),\qquad  \forall q\in W_h^0.\label{2}
\end{eqnarray}
\end{algorithm}

\section{Solution Existence and Uniqueness}
Recall that $\T_h$ is a shape-regular finite element partition of
the domain $\Omega$. For any $T\in\T_h$ and $\varphi\in H^{1}(T)$, the following trace inequality holds true \cite{wy3655}:
\begin{equation}\label{trace-inequality}
\|\varphi\|_{\pT}^2 \leq C
(h_T^{-1}\|\varphi\|_{T}^2+h_T\| \varphi\|_{1, T}^2).
\end{equation}
Furthermore, if $\varphi$ is a polynomial on $T$, the standard inverse inequality yields
\begin{equation}\label{trace}
\|\varphi\|_{\pT}^2 \leq Ch_T^{-1}\|\varphi\|_{T}^2.
\end{equation}
  
   For any $p=\{p_0, p_b\}\in W_h$, we define
 the following energy  norm 
  \begin{equation}\label{disnorma}
  \3bar p\3bar_{W_h}=\Big( \sum_{T\in {\cal T}_h} h_T\|p_0-p_b\|_{\partial T}^2\Big)^{\frac{1}{2}}.
 \end{equation}

For any $\bv=\{\bv_0, \bv_b\}\in V_h$, we define the energy norm  
\begin{equation}\label{3norm}
\3bar \bv\3bar_{V_h} =\Big(\sum_{T\in {\cal T}_h} \| \nabla_w \times \bv\|_T^2\Big)^{\frac{1}{2}},
\end{equation}
and the following discrete $H^1$ norm
\begin{equation}\label{disnorm}
\begin{split}
 \| \bv\|_{1,h} =&\Big(\sum_{T\in {\cal T}_h} \| \nabla \times\bv_0\|_T^2+h_T^{-1}\|\bv_0\times\bn-\bv_b\times\bn \|^2_{\partial T} \Big)^{\frac{1}{2}}.   
\end{split}
\end{equation}

 \begin{lemma} \label{normm}
 For $\bv=\{\bv_0, \bv_b\}\in V_h$, there exists a constant $C$ such that
 \begin{equation}\label{normeq1}
     \|\nabla\times \bv_0\|_T\leq C\|\nabla_w\times  \bv\|_T.
 \end{equation} 
\end{lemma}
\begin{proof} 
Let  $T\in {\cal T}_h$ be a polyhedral element with $N$  faces denoted by $e_1, \cdots, e_N$. It is important to emphasis that the polyhedral element $T$  can be non-convex. On each  face $e_i$, we construct   a linear equation  $l_i(x)$ such that  $l_i(x)=0$ on  face $e_i$ as follows: 
$$l_i(x)=\frac{1}{h_T}\overrightarrow{AX}\cdot \bn_i, $$  where  $A=(A_1,   A_{2})$ is a given point on the  face $e_i$,  $X=(x_1,   x_2)$ is any point on the  face $e_i$, $\bn_i$ is the normal direction to the  face $e_i$, and $h_T$ is the size of the element $T$. 

The bubble function of  the element  $T$ can be  defined as 
 $$
 \Phi_B =l^2_1(x)l^2_2(x)\cdots l^2_N(x) \in P_{2N}(T).
 $$ 
 It is straightforward to verify that  $\Phi_B=0$ on the boundary $\partial T$.    The function 
  $\Phi_B$  can be scaled such that $\Phi_B(M)=1$ where   $M$  represents the barycenter of the element $T$. Additionally,  there exists a sub-domain $\hat{T}\subset T$ such that $\Phi_B\geq \rho_0$ for some constant $\rho_0>0$.

For $\bv=\{\bv_0, \bv_b\}\in V_h$,   letting $\bq=\Phi_B  \nabla \times \bv_0$  in \eqref{discurlcurlnew} where we can specify $r=2N+k-1$,  yields:
\begin{equation} \label{t12}
\begin{split}
  &(\nabla_{w} \times \bv, \Phi_B  \nabla \times \bv_0)_T\\ =&  (  \nabla\times \bv_0,    \Phi_B  \nabla \times \bv_0)_T-\langle (\bv_b-\bv_0)\times\bn,  \Phi_B  \nabla \times \bv_0\rangle_{\partial T} 
\\=&(\nabla\times \bv_0,    \Phi_B  \nabla \times \bv_0)_T,
\end{split} 
\end{equation}
where we used $\Phi_B=0$ on $\partial T$.

From the domain inverse inequality \cite{wy3655},  there exists a constant $C$ such that
\begin{equation}\label{t2}
(\nabla\times  \bv_0,   \Phi_B   \nabla \times  \bv_0)_T \geq C (\nabla\times \bv_0,   \nabla \times \bv_0)_T.
\end{equation} 

Applying the Cauchy-Schwarz inequality and using \eqref{t12}-\eqref{t2}, we have:
 \begin{equation*}
     \begin{split}
      (\nabla\times \bv_0,    \nabla \times \bv_0)_T&\leq C (\nabla_{w} \times  \bv, \Phi_B \nabla \times \bv_0)_T \\& \leq C
\|\nabla_w \times \bv\|_T \| \Phi_B \nabla \times  \bv_0\|_T\\
& \leq C
\|\nabla_w \times  \bv\|_T \|\nabla \times \bv_0\|_T,
     \end{split}
 \end{equation*}
which gives
 $$
 \|\nabla\times \bv_0\|_T\leq C\|\nabla_w \times  \bv\|_T.
 $$

This completes the proof of the lemma.
\end{proof}

\begin{remark}
   If the polyhedral element $T$  is convex, 
   the bubble function of  the element  $T$ in Lemma \ref{normm}  can be  simplified to
 $$
 \Phi_B =l_1(x)l_2(x)\cdots l_N(x).
 $$ 
It can be verified that there exists a sub-domain $\hat{T}\subset T$,  such that
 $ \Phi_B\geq\rho_0$  for some constant $\rho_0>0$,  and $\Phi_B=0$ on the boundary $\partial T$.   Lemma \ref{normm}   can be proved in the same manner using this simplified construction. In this case, we take $r=N+k-1$.  
\end{remark}

  We construct a face-based bubble function   $$\varphi_{e_k}= \Pi_{i=1, \cdots, N, i\neq k}l_i^2(x).$$ It can be verified that  (1) $\varphi_{e_k}=0$ on the  face $e_i$ for $i \neq k$, (2) there exists a subdomain $\widehat{e_k}\subset e_k$ such that $\varphi_{e_k}\geq \rho_1$ for some constant $\rho_1>0$.

\begin{lemma}\label{phi}
     For $\{\bv_0, \bv_b\} \in V_h$, let  $\bq=(\bv_b-\bv_0)\times\bn \varphi_{e_k}$, where $\bn$ is the unit outward normal direction to the  face  $e_k$. The following inequality holds:
\begin{equation}
  \|\bq\|_T ^2 \leq Ch_T \int_{e_k}((\bv_b-\bv_0)\times\bn)^2ds.
\end{equation}
\end{lemma}
\begin{proof}
 We first extend $\bv_b$, initially defined on the two-dimensional  face  $e_k$, to the entire d-dimensional polyhedral element $T$  using  the following formula:
$$
 \bv_b (X)= \bv_b(Proj_{e_k} (X)),
$$
where $X=(x_1,\cdots,x_3)$ is any point in the  element $T$, $Proj_{e_k} (X)$ denotes the orthogonal projection of the point $X$ onto the  face  $e_k$.

We claim that $\bv_b$ remains  a polynomial defined on the element $T$ after the extension.  

Let the  face  $e_k$ be defined by two linearly independent vectors $\beta_1,  \beta_{2}$ originating from a point $A$ on the  face  $e_k$. Any point $P$ on the  face  $e_k$ can be parametrized as
$$
P(t_1,   t_{2})=A+t_1\beta_1+ t_{2}\beta_{2},
$$
where $t_1,   t_{2}$ are parameters.

Note that $\bv_b(P(t_1,  t_{2}))$ is a polynomial of degree $k$ defined on the  face  $e_k$. It can be expressed as:
$$
\bv_b(P(t_1, t_{2}))=\sum_{|\alpha|\leq k}c_{\alpha}\textbf{t}^{\alpha},
$$
where $\textbf{t}^{\alpha}=t_1^{\alpha_1}  t_{2}^{\alpha_{2}}$ and $\alpha=(\alpha_1,  \alpha_{2})$.

For any point $X=(x_1, \cdots, x_3)$ in the element $T$, the projection  $Proj_{e_k} (X)$ onto the  face  $e_k$ is the point on $e_k$ that minimizes the distance to $X$. Mathematically, this projection $Proj_{e_k} (X)$ is an affine transformation which can be expressed as 
$$
Proj_{e_k} (X)=A+\sum_{i=1}^{2} t_i(X)\beta_i,
$$
where $t_i(X)$ are the projection coefficients, and $A$ is the origin point on $e_k$. The coefficients $t_i(X)$ are determined  by solving the orthogonality condition:
$$
(X-Proj_{e_k} (X))\cdot \beta_j=0, \qquad \forall j=1, 2.
$$
This results in a system of linear equations in $t_1(X)$, $t_{2}(X)$, which  can be solved to yield:
$$
t_i(X)= \text{linear function of} \  X.
$$
Hence, the projection $Proj_{e_k} (X)$ is an affine linear function  of $X$.

We extend the polynomial $\bv_b$ from the  face  $e_k$ to the entire element $T$ by defining
$$
\bv_b(X)=\bv_b(Proj_{e_k} (X))=\sum_{|\alpha|\leq k}c_{\alpha}\textbf{t}(X)^{\alpha},
$$
where $\textbf{t}(X)^{\alpha}=t_1(X)^{\alpha_1}  t_{2}(X)^{\alpha_{2}}$. Since $t_i(X)$ are linear functions of $X$, each term $\textbf{t}(X)^{\alpha}$ is a polynomial in $X=(x_1, \cdots, x_3)$.
Thus, $\bv_b(X)$ is a polynomial in the three dimensional coordinates $X=(x_1, \cdots, x_3)$.

 Secondly, let $\bv_{trace}$ denote the trace of $\bv_0$ on the  face  $e_k$. We extend $\bv_{trace}$   to the entire element $T$  using  the following formula:
$$
 \bv_{trace} (X)= \bv_{trace}(Proj_{e_k} (X)),
$$
where $X$ is any point in the element $T$, $Proj_{e_k} (X)$ denotes the projection of the point $X$ onto the  face  $e_k$. Similar to the case for $\bv_b$, $\bv_{trace}$ remains a polynomial after this extension. 

Let $\bq=(\bv_b-\bv_0)\times\bn \varphi_{e_k}$. We have
\begin{equation*}
    \begin{split}
\|\bq\|^2_T  =
\int_T \bq^2dT =  &\int_T ((\bv_b-\bv_{trace})(X)\times\bn \varphi_{e_k})^2dT\\
\leq &Ch_T \int_{e_k} ((\bv_b-\bv_{trace})(Proj_{e_k} (X))\times\bn \varphi_{e_k})^2ds\\
\leq &Ch_T \int_{e_k} ((\bv_b-\bv_0)\times\bn)^2ds,
    \end{split}
\end{equation*} 
where we used the fact that there exists a subdomain $\widehat{e_k}\subset e_k$ such that $\varphi_{e_k}\geq \rho_1$ for some constant $\rho_1>0$, and applied the properties of the projection.

 This completes the proof of the lemma.

\end{proof}
\begin{lemma}\label{normeqva}   There exists two positive constants $C_1$ and $C_2$ such that for any $\bv=\{\bv_0, \bv_b\} \in V_h$, we have
 \begin{equation}\label{normeq}
 C_1\|\bv\|_{1, h}\leq \3bar \bv\3bar_{V_h}  \leq C_2\|\bv\|_{1, h}.
\end{equation}
\end{lemma} 

\begin{proof}    
Letting $\bq=(\bv_b-\bv_0)\times\bn \varphi_{e_k}$ in \eqref{discurlcurlnew} gives
  \begin{equation} \label{t33}
\begin{split}
&(\nabla_w \times  \bv, \bq)_T\\=&( \nabla\times  \bv_0,   \bq)_T-\langle (\bv_b-\bv_0)\times\bn,  \bq\rangle_{\partial T} \\
=&(\nabla\times \bv_0,   \bq)_T-C \int_{e_k} |(\bv_b-\bv_0)\times\bn|^2  \varphi_{e_k} ds.
\end{split} 
\end{equation}
Using Cauchy-Schwarz inequality, the domain inverse inequality \cite{wy3655}, Lemma \ref{phi} and \eqref{t33} gives
\begin{equation*}
\begin{split}
  \int_{e_k}|(\bv_b-\bv_0)\times\bn|^2 ds\leq & C \int_{e_k} |(\bv_b-\bv_0)\times\bn|^2  \varphi_{e_k}  ds
  \\ \leq& C (\|\nabla_w \times  \bv\|_T+\|\nabla\times \bv_0\|_T){ \|\bq\|_T}\\
 \leq & C { h_T^{\frac{1}{2}}} (\| \nabla_w \times  \bv\|_T+\| \nabla\times \bv_0\|_T){ (\int_{e_k} |(\bv_b-\bv_0)\times\bn|^2ds)^{\frac{1}{2}}},
 \end{split}
\end{equation*}
which, by using \eqref{normeq1}, gives 
\begin{equation}\label{t21}
 h_T^{-1}\int_{e_k} |(\bv_b-\bv_0)\times\bn|^2ds \leq C  (\| \nabla_w \times  \bv\|^2_T+\| \nabla\times  \bv_0\|^2_T)\leq C\| \nabla_w \times  \bv\|^2_T.   
\end{equation}
 
  Using     \eqref{3norm}- \eqref{normeq1}  and \eqref{t21}, we get
$$
 C_1\|\bv\|_{1, h}\leq \3bar \bv\3bar_{V_h}.
$$

Next, applying the Cauchy-Schwarz inequality and  the trace inequality \eqref{trace} to  \eqref{discurlcurlnew}, gives 
\begin{equation*} 
\begin{split}
& \Big| (\nabla_w \times  \bv, \bq)_T \Big| \\
\leq &\|\nabla\times \bv_0\|_T \|\bq\|_T+\|(\bv_b-\bv_0)\times\bn\|_{\partial T}\|  \bq\|_{\partial T} \\ 
\leq &\|\nabla\times  \bv_0\|_T \|\bq\|_T+Ch_T^{-\frac{1}{2}}\|(\bv_b-\bv_0)\times\bn\|_{\partial T}\|\bq\|_{T}. 
\end{split} 
\end{equation*}
This yields
$$
\|\nabla_w \times \bv\|_T^2\leq C( \| \nabla\times  \bv_0\|^2_T  +
 h_T^{-1}\|(\bv_b-\bv_0)\times\bn\|^2_{\partial T}),
$$
 and further gives $$ \3bar \bv\3bar_{V_h}  \leq C_2\|\bv\|_{1, h}.$$

 This completes the proof of the lemma. 
 \end{proof}

 \begin{remark}
   If the polyhedral element $T$  is convex, 
  the   face-based bubble function in Lemma \ref{phi} and Lemma \ref{normeqva}   can be  simplified to
$$\varphi_{e_k}= \Pi_{i=1, \cdots, N, i\neq k}l_i(x).$$
It can be verified that (1) $\varphi_{e_k}=0$ on the  face $e_i$ for $i \neq k$, (2) there exists a subdomain $\widehat{e_k}\subset e_k$ such that $\varphi_{e_k}\geq \rho_1$ for some constant $\rho_1>0$. Lemma \ref{phi} and Lemma \ref{normeqva}   can be proved in the same manner using this simplified construction.  
\end{remark}

\begin{theorem}
The weak Galerkin finite element scheme (\ref{32})-(\ref{2}) has a unique solution.
\end{theorem}
\begin{proof}
It suffices to prove that if $\bf=0$,  $\bphi=0$, and $g=0$, then  $\bu_h=0$ and $p_h=0$ in $\Omega$. 
To this end, taking $\bv=\bu_h$ in (\ref{32}) and $q=p_h$ in (\ref{2}) gives
$$
\sum_{T\in {\cal T}_h} (\nabla_w\times\bu_h, \nabla_w\times \bu_h)_T+ h_T\langle p_0-p_b, p_0-p_b\rangle_{\partial T}=0.
$$
This implies 
$$\3bar\bu_h\3bar_{V_h}=0, \qquad p_0=p_b  \ \text{on each} \ \partial T,
$$ 
which, together with  \eqref{normeq}, yields 
$$
\|\bu_h\|_{1, h}=0, \qquad p_0=p_b \  \text{on each} \ \partial T.
$$ 
This gives
\begin{eqnarray} 
 \nabla\times \bu_0&=&0, \quad \text{on each}\ T,\label{tt1}
  \\
   \bu_0 \times \bn&=&\bu_b\times \bn, \  \text{on each}\ \partial T,\label{tt3}
      \\
     p_0&=&p_b, \quad \text{on each}\ \partial T.\label{tt5}
 \end{eqnarray}

Using \eqref{tt5} implies $s(p_h, q)=0$ in \eqref{2}. Thus,  \eqref{2} indicates 
\begin{equation}\label{ee1}
\begin{split}
0&=\sum_{T\in {\cal T}_h}  (\nabla_w q, \bu_0)_T\\
&=\sum_{T\in {\cal T}_h} -(q_0, \nabla\cdot\bu_0)_T+\langle q_b, \bu_0\cdot\bn\rangle_{\pT}\\
&=\sum_{T\in {\cal T}_h}  -(q_0, \nabla\cdot\bu_0)_T+\sum_{e\in {\cal E}_h^0}\langle q_b, \ljump \bu_0\cdot\bn\rjump \rangle_{e},
\end{split}
\end{equation}
where we used \eqref{disgradient} and the fact that  $q_b=0$ on $\partial\Omega$, and $\ljump \bu_0\cdot\bn\rjump$ denotes the jump of $\bu_0\cdot\bn$ on the interior  face $e\in {\cal E}_h^0$.
Letting $q_0=0$ and $q_b=\ljump \bu_0\cdot\bn \rjump$ in \eqref{ee1} yields that $\ljump \bu_0\cdot\bn \rjump=0$ on $e\in {\cal E}_h^0$, which means $\bu_0\cdot\bn$ is continuous along the interior interface $e\in {\cal E}_h^0$. This follows that $\bu_0\in H(div; \Omega)$. Taking $q_0=\nabla\cdot \bu_0$ and $q_b=0$  in \eqref{ee1} gives $\nabla \cdot \bu_0=0$ on each $T$. This, together with the fact that $\bu_0\in H(div; \Omega)$, yields $\nabla \cdot \bu_0=0$ in $\Omega$.

It follows from  \eqref{tt3} that $\bu_0\times\bn$  is continuous across the interior interface ${\cal E}_h^0$. Thus, $\bu_0\in H(curl; \Omega)$. This, together with \eqref{tt1}, gives $\nabla \times \bu_0=0$ in $\Omega$.  Thus, there exists a potential function $\psi$ such that $\bu_0=\nabla \psi$ in $\Omega$. It follows from the fact that  $\nabla \cdot\bu_0=0$ in $\Omega$ that $\nabla\cdot\bu_0=\Delta \psi=0$ strongly holds in $\Omega$ with the boundary condition $\nabla \psi \times \bn=\bu_0\times \bn=\bu_b\times \bn=0$ on $\partial \Omega$, where we used \eqref{tt3} and $\bu_b\times \bn=0$ on $\partial \Omega$. This implies that $\psi=C$ on $\partial\Omega$. The uniqueness of the solution of the Laplace equation implies that $\psi=C$  is the only solution of $\Delta \psi=0$.   Thus, $\bu_0=\nabla \psi\equiv 0$ in $\Omega$. Using  \eqref{tt3} gives $\bu_b \times \bn=\bu_0 \times \bn= 0$ in $\Omega$, which indicates $\bu_b=0$ and further $\bu_h\equiv 0$ in $\Omega$.

Since $\bu_h\equiv 0$ in $\Omega$,   \eqref{32} implies that $b(\bv, p_h)=0$ for any $\bv\in V_h^0$. It follows from \eqref{disgradient*} and \eqref{tt5} that 
\begin{equation*}
    \begin{split}
        0=&b(\bv, p_h)=(\bv_0, \nabla_w p_h)\\
        = &\sum_{T\in {\cal T}_h}  (\nabla  p_0, \bv_0)_T-\langle p_0-p_b, \bv_0\cdot\bn \rangle_{\partial T}\\
        =&\sum_{T\in {\cal T}_h}  (\nabla  p_0, \bv_0)_T.
    \end{split}
\end{equation*}
Letting $\bv=\{\bv_0, \bv_b\}=\{\nabla p_0, 0\}$ in the above equation gives $\nabla p_0=0$ on each $T\in {\cal T}_h$. Thus, $p_0=C$ on each $T$, which, together with $\eqref{tt5}$, gives $p_0=C$ in the domain $\Omega$. This, together with $p_b=0$ on $\partial\Omega$ and \eqref{tt5}, gives $p_0\equiv 0$ in $\Omega$. This, together with \eqref{tt5}, yields  $p_b\equiv 0$ in $\Omega$. Thus,   $p_h\equiv 0$ in $\Omega$.

This completes the proof of the theorem. 
\end{proof}

\section{Error Equations}\label{Section:error-equation}
In this section, we derive the error equations for the weak Galerkin method \eqref{32}-\eqref{2} applied to the Maxwell equations \eqref{model4}, based on the weak formulation \eqref{weakform}. These error equations are instrumental for the subsequent convergence analysis.

Let  $k\geq 1$. Let $\bQ_0$ be the $L^2$ projection operator onto $[P_k(T)]^3$. Analogously, for $e\subset\partial T$, denote by $\bQ_b$  the $L^2$ projection operator  onto $[P_{k}(e)]^3$. For $\bw\in H(curl; \Omega)$, define the $L^2$ projection $\bQ_h \bw\in V_h$ as follows
$$
\bQ_h\bw|_T=\{\bQ_0 \bw, \bQ_b \bw\}.
$$
For $\sigma \in H^1(\Omega)$, the $L^2$ projection $Q_h \sigma\in W_h$ is defined by 
 $$
 Q_h\sigma|_T=\{Q_0 \sigma, Q_b \sigma\},
$$
where $Q_0$ and $Q_b$ are the $L^2$ projection operators onto $P_{k-1}(T)$ and $P_k(e)$ respectively.  Denote by ${\cal Q}_h^{r}$  the $L^2$ projection operators onto $P_{r}(T)$.

\begin{lemma} 
    The following properties hold true:
\begin{equation}\label{l}
    \nabla_w\times \bu ={\cal Q}_h^{r}(\nabla\times \bu), \qquad \forall\bu\in H(curl; T),    
\end{equation}
 \begin{equation}\label{l-2}
    \nabla_w (Q_hq)=\bQ_0 (\nabla q), \qquad \forall q\in H^1(T),  
\end{equation}      
\begin{equation}\label{l-3}
    \nabla_w \cdot \bu={\cal Q}_h^{r}(\nabla \cdot\bu), \qquad \forall\bu\in H(div; T).   
\end{equation}  
\end{lemma}

\begin{proof}
For any $\bu\in H(curl; T)$, it follows from \eqref{discurlcurlnew} that 
\begin{equation*}  
\begin{split}
&( \nabla_{w} \times   \bu, \bq)_T\\=&(\nabla\times \bu,   \bq)_T-\langle (\bw|_{\partial T}-\bw|_T)\times\bn,  \bq\rangle_{\partial T} 
\\=&(\nabla\times \bu,   \bq)_T\\
=&({\cal Q}_h^{r}(\nabla\times \bu), \bq)_T,
\end{split} 
\end{equation*}
for any $\bq\in [P_{r}(T)]^3$. 
This completes the proof of \eqref{l}. 

For any $q\in H^1(T)$, using \eqref{disgradient} gives
\begin{equation*} 
 \begin{split}
     (\nabla_{w} (Q_hq), \boldsymbol{\psi})_T=& -(Q_0q, \nabla \cdot\boldsymbol{\psi})_T+\langle  Q_bq, \boldsymbol{\psi} \cdot \textbf{n}\rangle_{\partial T}\\
    =& -(q, \nabla \cdot\boldsymbol{\psi})_T+\langle  q, \boldsymbol{\psi} \cdot \textbf{n}\rangle_{\partial T}\\
    =& (\nabla q, \boldsymbol{\psi})_T 
 \\=& (\bQ_0(\nabla q), \boldsymbol{\psi})_T,
    \end{split} 
 \end{equation*}
 for any $\boldsymbol{\psi}\in [P_k (T)]^3$. This completes the proof of \eqref{l-2}.

 For any $\bu\in H(div; T)$, it follows from \eqref{disdivnew} that 
\begin{equation*}  
\begin{split}
&( \nabla_{w} \cdot \bu, q)_T\\=&(\nabla \cdot\bu,   q)_T+\langle (\bw|_{\partial T}-\bw|_T) \cdot\bn,  q\rangle_{\partial T} 
\\=&(\nabla\cdot \bu,   q)_T\\
=&({\cal Q}_h^{r}(\nabla\cdot \bu), q)_T,
\end{split} 
\end{equation*}
for any $q\in P_{r}(T)$. 
This completes the proof of \eqref{l-3}. 

\end{proof}

Let $(\bu_h, p_h)$ be the WG solutions of \eqref{32}-\eqref{2}. Define the error functions $\be_h$ and $\epsilon_h$  as follows: 
\begin{eqnarray}\label{error}
\be_h&=&\bu-
\bu_h,\\
\epsilon_h&=&Q_hp-p_h. \label{error-2}
\end{eqnarray}

 \begin{lemma}\label{errorequa}
Let $(\bu; p)$  be the exact solutions of the Maxwell equations \eqref{model4}, and let  $(\bu_h; p_h) \in V_h\times W_h^0$  be the numerical solutions  obtained from WG scheme \eqref{32}-\eqref{2}.    The error functions $\be_h$ and $\epsilon_h$ defined in \eqref{error}-\eqref{error-2} satisfy the following error equations: 
\begin{eqnarray}\label{sehv}
a(\be_h, \bv)-b(\bv, \epsilon_h)&=&  \ell_1(\bu, \bv),\quad   \forall \bv\in V_{h}^0,\\
s(\epsilon_h, q)+b(\be_h, q)&=& s(Q_hp, q) +\ell_2(\bu, q),\quad\forall q\in W^0_h. \label{sehv2}
\end{eqnarray}
\end{lemma}
Here
 \begin{eqnarray*}\label{lu}
\ell_1(\bu, \bv)&=&\sum_{T\in {\cal T} _h}\langle \nu (\bv_b-\bv_0)\times \bn, (I-{\cal Q}_h^{r})\nabla\times \bu\rangle_{\partial T},\\
  \ell_2(\bu, q)&=&  \sum_{T\in{\cal T}_h} \langle q_0-q_b, (I-\bQ_0)\bu \cdot\bn\rangle_{\pT}.
\end{eqnarray*}
 
\begin{proof} Letting $\bq={\cal Q}_h^{r} (\nabla\times\bu)$ in \eqref{discurlcurlnew} and  using \eqref{l}, we have 
\begin{equation}\label{eq1}
\begin{split}
&\sum_{T\in {\cal T}_h}(\nu\nabla_w\times \bu,  \nabla_w\times  \bv)_T\\=& \sum_{T\in {\cal T}_h}(\nu{\cal Q}_h^{r} (\nabla\times \bu), \nabla_w\times \bv)_T\\
=& \sum_{T\in {\cal T}_h}  (\nu  \nabla\times \bv_0,   {\cal Q}_h^{r} (\nabla\times \bu))_T-\langle \nu(\bv_b-\bv_0)\times\bn,  {\cal Q}_h^{r} (\nabla\times \bu)\rangle_{\partial T}\\
=& \sum_{T\in {\cal T}_h}  ( \nu \nabla\times \bv_0,   \nabla\times \bu)_T-\langle \nu(\bv_b-\bv_0)\times\bn,  {\cal Q}_h^{r} (\nabla\times \bu)\rangle_{\partial T}.
\end{split}
\end{equation}

It follows from \eqref{l-2} that
\begin{equation}\label{eqq3}
\begin{split}
b(\bv, Q_hp)=\sum_{T\in {\cal T}_h}(\nabla_w(Q_hp), \bv_0)_T=\sum_{T\in {\cal T}_h}(\bQ_0(\nabla p), \bv_0)_T=\sum_{T\in {\cal T}_h}(\nabla p, \bv_0)_T.
\end{split}
\end{equation}
Testing the first equation in \eqref{model4}  with  $\bv_0$  where  $\bv=\{\bv_0, \bv_b\}\in V_h^0$ gives
\begin{equation}\label{eqq}
    (\nabla\times (\nu \nabla\times\bu), \bv_0)-(\nabla p, \bv_0)=(\bf, \bv_0).
\end{equation}
Using the usual integration by parts gives
\begin{equation}\label{eqq1}
    \begin{split}
\sum_{T\in {\cal T} _h}(\nabla\times(\nu \nabla\times \bu), \bv_0)_T=&\sum_{T\in {\cal T} _h}(\nu \nabla\times \bu, \nabla\times \bv_0)_T\\&-\langle \nu (\bv_b-\bv_0)\times \bn, \nabla\times \bu\rangle_{\partial T},        
    \end{split}
\end{equation}
where we used the fact that $\sum_{T\in {\cal T} _h}\langle \nu  \bv_b \times \bn, \nabla\times \bu\rangle_{\partial T}=\langle \nu  \bv_b \times \bn, \nabla\times \bu\rangle_{\partial \Omega}=0$ due to $\bv_b \times \bn=0$ on $\partial\Omega$. 

From \eqref{eq1} and \eqref{eqq1}, we have
\begin{equation}\label{eqq2}
    \begin{split}
\sum_{T\in {\cal T} _h}(\nabla\times(\nu \nabla\times \bu), \bv_0)_T=&
  \sum_{T\in {\cal T} _h}(\nu\nabla_w\times \bu,  \nabla_w\times  \bv)_T\\&+  \langle \nu (\bv_b-\bv_0)\times \bn, ({\cal Q}_h^{r}-I)\nabla\times \bu\rangle_{\partial T}.  \end{split}
\end{equation}
Substituting \eqref{eqq2} and \eqref{eqq3} into \eqref{eqq} yields
\begin{equation}\label{eqq4}
     \begin{split}
    & \sum_{T\in {\cal T} _h}(\nu\nabla_w\times \bu,  \nabla_w\times  \bv)_T -(\nabla_w (Q_hp),\bv_0)_T \\=&\sum_{T\in {\cal T} _h}(\bf, \bv_0)_T-  \langle \nu (\bv_b-\bv_0)\times \bn, ({\cal Q}_h^{r}-I)\nabla\times \bu\rangle_{\partial T}.
\end{split}
\end{equation}
Subtracting  \eqref{32} from \eqref{eqq4} gives
\begin{equation*}
    \begin{split}
 &\sum_{T\in {\cal T} _h}(\nu\nabla_w\times \be_h,  \nabla_w\times  \bv)_T -(\nabla_w \epsilon_h,\bv_0)_T \\=& -  \sum_{T\in {\cal T} _h}\langle \nu (\bv_b-\bv_0)\times \bn, ({\cal Q}_h^{r}-I)\nabla\times \bu\rangle_{\partial T}.       
    \end{split}
\end{equation*}   
 
  To derive \eqref{sehv2},  we test the second equation in \eqref{model4} with  $q_0$ where $q=\{q_0, q_b\}\in W_h^0$ and use the usual integration by parts to obtain
\begin{equation}\label{eq5}
 \sum_{T\in {\cal T}_h}- (\bu, \nabla q_0)+  \langle \bu\cdot\bn, q_0-q_b\rangle_{\pT}=\sum_{T\in {\cal T}_h}(g, q_0)_T,
\end{equation}
where we used $\sum_{T\in {\cal T}_h} \langle \bu\cdot\bn,  q_b\rangle_{\pT}=\langle \bu\cdot\bn,  q_b\rangle_{\partial\Omega}=0$ due to $q_b=0$ on $\partial\Omega$.

Using \eqref{disgradient*} and the usual integration by parts gives
\begin{equation*} \label{beq}
\begin{split}
s(Q_hp,q)+b(\bu, q) =&s(Q_hp,q)+\sum_{T\in {\cal T}_h}   (\bu, \nabla_w q)_T\\
=&s(Q_hp,q)+\sum_{T\in {\cal T}_h}   (\bQ_0\bu, \nabla_w q)_T\\
=& s(Q_hp,q)+\sum_{T\in {\cal T}_h} (\nabla q_0, \bQ_0\bu)_T-\langle q_0-q_b, \bQ_0\bu\cdot\bn\rangle_{\pT}\\
=& s(Q_hp,q)+\sum_{T\in {\cal T}_h} (\nabla q_0,  \bu)_T-\langle q_0-q_b, \bQ_0\bu\cdot\bn\rangle_{\pT}\\=& s(Q_hp,q)+ \sum_{T\in {\cal T}_h}\langle q_0-q_b, (I-\bQ_0)\bu\cdot\bn\rangle_{\pT}-(g, q_0)_T,
\end{split}
\end{equation*}
where we used \eqref{eq5} on the last line. 

Subtracting \eqref{2} from the above equation completes the proof of 
\eqref{sehv2}.

This completes the proof of the lemma. 
\end{proof}

\section{Error Estimates}\label{Section:error-estimates}
This section is dedicated to deriving the error estimates in a discrete 
$H^1$ norm for the weak Galerkin   scheme \eqref{32}-\eqref{2} when applied to the Maxwell equations \eqref{model4}. 
 
\begin{lemma}\cite{yemaxwell, wangmaxwell}
 Let $0\leq r\leq k$ and $0\leq s\leq 1$.  Assume that $\bv\in [H^{r+1}(\Omega)]^3$ and $q\in H^r(\Omega)$.  There holds
 \begin{equation}\label{est1}
     \sum_{T\in {\cal T}_h} h_T^{2s}\|\bv-\bQ_0\bv\|_{s, T}^2\leq Ch^{2(r+1)}\|\bv\|^2_{r+1},
 \end{equation}
 \begin{equation}\label{est2}
     \sum_{T\in {\cal T}_h} h_T^{2s}\|\nabla\times \bv-{\cal Q}_h^r(\nabla\times \bv)\|_{s, T}^2\leq Ch^{2r}\|\bv\|^2_{r+1}, 
 \end{equation}
 \begin{equation}\label{est4}
     \sum_{T\in {\cal T}_h} h_T^{2s}\|q-Q_0q\|_{s, T}^2\leq Ch^{2r}\|q\|^2_{r}. 
 \end{equation}

\end{lemma}

\begin{lemma}\label{lem2}  Let  $k\geq 1$.
Suppose $\bu\in [H^{k+1}(\Omega)]^3$ and $p\in H^{k}(\Omega)$. Then, for $(\bv, q)\in V_h\times W_h$, the following estimates hold true; i.e.,
\begin{eqnarray}\label{error2}
|\ell_1(\bu, \bv)|&\leq&    Ch^k \| \bu\|_{k+1}\|\bv\|_{1,h},\\
 |\ell_2(\bu, q)| &\leq&  Ch^{k}\|\bu\|_{k+1}\3bar q\3bar_{W_h},\label{error3} \\
 |s(Q_hp, q)| &\leq&  Ch^{k}\|p\|_k \3bar q\3bar_{W_h}. \label{error4} 
\end{eqnarray}
\end{lemma}
\begin{proof} 
Using the Cauchy-Schwarz inequality, the trace inequality \eqref{trace-inequality} and the estimate \eqref{est2} with $r=k$ and $s=0, 1$, we obtain: 
\begin{equation*}
\begin{split}
&|\ell_1(\bu, \bv)| \\=&|\sum_{T\in {\cal T} _h}\langle \nu (\bv_b-\bv_0)\times \bn, (I-{\cal Q}_h^{r})\nabla\times \bu\rangle_{\partial T}|
\\
\leq &  \big(\sum_{T\in{\cal T}_h}h_T^{-1}\| \nu (\bv_b-\bv_0)\times \bn\|^2_{\partial T}\big)^{\frac{1}{2}} \big(\sum_{T\in{\cal T}_h}h_T \| (I-{\cal Q}_h^{r})\nabla\times \bu\|^2_{\partial T}\big)^{\frac{1}{2}}  
\\
\leq & C \|\bv\|_{1,h} \big(\sum_{T\in{\cal T}_h}  \| (I-{\cal Q}_h^{r})\nabla\times \bu\|^2_{T}+h_T^2 \| (I-{\cal Q}_h^{r})\nabla\times \bu\|^2_{1, T}\big)^{\frac{1}{2}} \\
\leq &  Ch^{r} \| \bu\|_{r+1} \|\bv\|_{1,h}.
\end{split}
\end{equation*}

Using the Cauchy-Schwarz inequality, the trace inequality \eqref{trace-inequality} and the estimate \eqref{est1} with $r=k$  and $s=0, 1$  gives
\begin{equation*}
\begin{split}
& |\ell_2(\bu, q)|\\
 \leq & \Big(\sum_{T\in{\cal T}_h}  h_T  \|q_0-q_b\|_{\pT}^2\Big)^{\frac{1}{2}}\Big(\sum_{T\in{\cal T}_h}h_T^{-1}  \|(I-\bQ_0)\bu \cdot\bn\|_{\pT}^2\Big)^{\frac{1}{2}}\\
 \leq & \Big(\sum_{T\in{\cal T}_h} h_T^{-2} \|(I-\bQ_0)\bu \cdot\bn\|_{T}^2+  \|(I-\bQ_0)\bu \cdot\bn\|_{1,T}^2\Big)^{\frac{1}{2}}\3bar q\3bar_{W_h} \\
 \leq &  Ch^{k}\|\bu\|_{k+1}\3bar q\3bar_{W_h}.
  \end{split}
\end{equation*}

It follows from the Cauchy-Schwarz inequality, the trace inequality \eqref{trace-inequality} and the estimate \eqref{est4} with $r=k$  and $s=0, 1$ that
 \begin{equation*}
\begin{split}
&|s(Q_hp,q)| \\
 \leq &|\sum_{T\in {\cal T}_h}h_T\langle Q_0p-Q_bp, q_0-q_b \rangle_{\partial T}| \\
 \leq &\Big(\sum_{T\in{\cal T}_h}  h_T  \|Q_0p-Q_bp\|_{\pT}^2\Big)^{\frac{1}{2}}\Big(\sum_{T\in{\cal T}_h}h_T   \| q_0-q_b\|_{\pT}^2\Big)^{\frac{1}{2}}\\
 \leq & \Big(\sum_{T\in{\cal T}_h}   \|Q_0p- p\|_{T}^2+h_T^2  \|Q_0p- p\|_{ 1,T}^2\Big)^{\frac{1}{2}}\3bar q\3bar_{W_h}\\
  \leq & Ch^{k}\|p\|_k \3bar q\3bar_{W_h}.
  \end{split}
\end{equation*}

This completes the proof of the lemma.
\end{proof}

\begin{lemma}\label{infsup} 
    For any $q=\{q_0, q_b\}\in W_h$, there exists a $\bv_q\in V_h^0$ such that 
 \begin{equation}\label{inf1}
b(\bv_q, q)\geq C_1 \sum_{T\in {\cal T}_h}(  q_0,  q_0)_T-C_2s(q,q),        
    \end{equation}
     \begin{equation}\label{inf2}
\3bar\bv_q\3bar_{V_h}\leq  C  (\sum_{T\in {\cal T}_h}(q_0,  q_0)_T)^{\frac{1}{2}}.   
    \end{equation}
\end{lemma}
\begin{proof} 
Using \eqref{disgradient*} and \eqref{disdivnew} gives
\begin{equation}\label{ss} 
\begin{split}
b(\bv,q)&= \sum_{T\in {\cal T}_h}(\bv_0, \nabla_w q)_T\\
&= \sum_{T\in {\cal T}_h}-(q_0, \nabla \cdot \bv_0)_T+\langle q_b, \bv_0\cdot\bn\rangle_{\partial T}\\
&= \sum_{T\in {\cal T}_h}-(q_0, \nabla_w \cdot \bv)_T+\langle q_0,  (\bv_b-\bv_0)\cdot\bn \rangle_{\partial T}+\langle q_b,  \bv_0\cdot\bn\rangle_{\partial T}\\
&= \sum_{T\in {\cal T}_h}-(q_0, \nabla_w \cdot \bv)_T+\langle q_0-q_b,  (\bv_b-\bv_0)\cdot\bn \rangle_{\partial T},\\
\end{split} 
\end{equation}
where we used $ \sum_{T\in {\cal T}_h} \langle q_b,  \bv_b\cdot\bn\rangle_{\partial T}=\langle q_b,  \bv_b\cdot\bn\rangle_{\partial \Omega}=0$  since $\bv_b=0$ on $\partial\Omega$.

Let $\bw$ be the solution of
\begin{equation}\label{eqs}
 \nabla \cdot\bw=-q_0.
   \end{equation}

Letting $\bv_q=\bQ_h\bw=\{\bQ_0 \bw, \bQ_b\bw\}$ in \eqref{ss}  and using \eqref{l-3}, \eqref{eqs},  Cauchy-Schwartz inequality, the trace inequality \eqref{trace-inequality} and Young's inequality  gives
\begin{equation*} 
\begin{split}
&b(\bv_q,q) \\
 = &\sum_{T\in {\cal T}_h}-(q_0, \nabla_w \cdot \bQ_h\bw)_T+\langle q_0-q_b,  (\bQ_b\bw-\bQ_0\bw)\cdot\bn \rangle_{\partial T} \\ =& \sum_{T\in {\cal T}_h}-(q_0, {\mathcal Q}_h^r(\nabla  \cdot  \bw))_T+\langle q_0-q_b,  (\bQ_b\bw-\bQ_0\bw)\cdot\bn \rangle_{\partial T} \\
 =& \sum_{T\in {\cal T}_h} (q_0, {\mathcal Q}_h^r q_0)_T+\langle q_0-q_b,  (\bQ_b\bw-\bQ_0\bw)\cdot\bn \rangle_{\partial T} \\
  \geq &    \sum_{T\in {\cal T}_h}  ( q_0,  q_0)_T-C (\sum_{T\in {\cal T}_h} h_T\|q_0-q_b\|_{\partial T}^2)^{\frac{1}{2}}\\&\cdot  (\sum_{T\in {\cal T}_h}h_T^{-1}\|(\bQ_b\bw-\bQ_0\bw)\cdot\bn \|_{\partial T}^2)^{\frac{1}{2}}\\
   \geq &      \sum_{T\in {\cal T}_h}  ( q_0,  q_0)_T-C (\sum_{T\in {\cal T}_h} h_T\|q_0-q_b\|_{\partial T}^2)^{\frac{1}{2}} \\&\cdot (\sum_{T\in {\cal T}_h}h_T^{-2}\|(\bQ_b\bw-\bQ_0\bw)\cdot\bn \|_{  T}^2 + \|(\bQ_b\bw-\bQ_0\bw)\cdot\bn \|_{1,  T}^2)^{\frac{1}{2}}
  \end{split} 
\end{equation*}
and
\begin{equation}\label{ssw} 
\begin{split}
&b(\bv_q,q) \\
 \geq &      \sum_{T\in {\cal T}_h}  ( q_0,  q_0)_T-C (\sum_{T\in {\cal T}_h} h_T\|q_0-q_b\|_{\partial T}^2)^{\frac{1}{2}}  \\&\cdot(\sum_{T\in {\cal T}_h}h_T^{-2}\| \bw-\bQ_0\bw  \|_{  T}^2+ \| \bw-\bQ_0\bw  \|_{1,  T}^2)^{\frac{1}{2}}\\
 \geq   &    \sum_{T\in {\cal T}_h}  ( q_0,  q_0)_T-C (\sum_{T\in {\cal T}_h} h_T\|q_0-q_b\|_{\partial T}^2)^{\frac{1}{2}}  (\sum_{T\in {\cal T}_h} \|\bw\|_{1,T}^2)^{\frac{1}{2}}  \\
 \geq  &     \sum_{T\in {\cal T}_h}  ( q_0,  q_0)_T-C (\sum_{T\in {\cal T}_h} h_T\|q_0-q_b\|_{\partial T}^2)^{\frac{1}{2}} (\sum_{T\in {\cal T}_h} ( q_0,  q_0)_T)^{\frac{1}{2}}\\ 
 \geq   &    \sum_{T\in {\cal T}_h}  ( q_0,  q_0)_T- \frac{C}{\epsilon} \sum_{T\in {\cal T}_h} h_T\|q_0-q_b\|^2_{\partial T}  -C\epsilon \sum_{T\in {\cal T}_h} ( q_0,  q_0)_T\\
 \geq  &   (1-C\epsilon)  \sum_{T\in {\cal T}_h}  ( q_0,  q_0)_T- \frac{C}{\epsilon} \sum_{T\in {\cal T}_h} h_T\|q_0-q_b\|^2_{\partial T}  \\
 \geq &     C_1 \sum_{T\in {\cal T}_h}  ( q_0,  q_0)_T- C_2\sum_{T\in {\cal T}_h} h_T\|q_0-q_b\|^2_{\partial T},
\end{split} 
\end{equation}
 where we can choose $\epsilon$   such that $1-C\epsilon>C_1$  and $\frac{C}{\epsilon}<C_2$ for some constants $C$, $C_1$, $C_2$.

Regarding to \eqref{inf2}, using \eqref{discurlcurlnew}, the trace inequality \eqref{trace} and Cauchy-Schwartz inequality gives
\begin{equation}\label{eqss}
    \|\nabla_w\times \bv\|^2_T\leq \|\nabla\times \bv_0\|^2_T+Ch_T^{-1}\|\bv_0-\bv_b\|^2_{\partial T}.
\end{equation}
Let $\bw$ be the solution of \eqref{eqs}.
Letting $\bv_q=\bQ_h\bw=\{\bQ_0 \bw, \bQ_b\bw\}$ in \eqref{eqss} gives
\begin{equation*}
\begin{split}
   \sum_{T\in {\cal T}_h} \|\nabla_w\times \bv_q\|^2_T\leq &   \sum_{T\in {\cal T}_h}\|\nabla\times (\bQ_0 \bw)\|^2_T+   Ch_T^{-1}\|\bQ_0 \bw-\bQ_b \bw\|^2_{\partial T}\\
   \leq &   \sum_{T\in {\cal T}_h}\|\nabla\times (\bQ_0 \bw)\|^2_T+Ch_T^{-2}\|\bQ_0 \bw-\bQ_b \bw\|^2_{T}+C \|\bQ_0 \bw-\bQ_b \bw\|^2_{1,T}  \\
    \leq &   \sum_{T\in {\cal T}_h}\| \bQ_0 \bw \|^2_{1,T}+Ch_T^{-2}\|\bQ_0 \bw- \bw\|^2_{T}+C \|\bQ_0 \bw- \bw\|^2_{1,T}  \\ 
     \leq &   \sum_{T\in {\cal T}_h}\|   \bw \|^2_{1,T}   
          \leq     \sum_{T\in {\cal T}_h}\|   q \|^2_{0,T},   
\end{split}
\end{equation*}
where we used the trace inequality \eqref{trace-inequality} and \eqref{eqs}.  This completes the proof for \eqref{inf2}.

\end{proof}

\begin{lemma}
Assume  the exact solution $\bu$ of the Maxwell equation \eqref{model} is sufficiently regular such that $\bu\in [H^{k+1} (\Omega)]^3$. There exists a constant $C$, such that the following estimate holds true; i.e.,
\begin{equation}\label{errorestinew}
\3bar \bu-\bQ_h\bu \3bar_{V_h} \leq Ch^k\|\bu\|_{k+1}.
\end{equation}
\end{lemma}
\begin{proof}
Using \eqref{discurlcurlnew}, Cauchy-Schwarz inequality, the trace inequalities \eqref{trace-inequality}-\eqref{trace}, we have  
\begin{equation*}
\begin{split}
&\quad\sum_{T\in {\cal T}_h}(\nabla_w\times(\bu-\bQ_h\bu), \bv)_T\\
&=\sum_{T\in {\cal T}_h}(\nabla\times (\bu-\bQ_0\bu),  \bv)_T-\langle (\bQ_0\bu-\bQ_b\bu)\times \bn, \bv \rangle_{\partial T}\\
&\leq \Big(\sum_{T\in {\cal T}_h}\|\nabla\times (\bu-\bQ_0\bu)\|^2_T\Big)^{\frac{1}{2}} \Big(\sum_{T\in {\cal T}_h}\|\bv\|_T^2\Big)^{\frac{1}{2}}\\&\quad+ \Big(\sum_{T\in {\cal T}_h} \|\bQ_0\bu-\bQ_b\bu\|_{\partial T} ^2\Big)^{\frac{1}{2}}\Big(\sum_{T\in {\cal T}_h} \|\bv\|_{\partial T}^2\Big)^{\frac{1}{2}}\\
&\leq\Big(\ \sum_{T\in {\cal T}_h} \|\nabla\times (\bu-\bQ_0\bu)\|_T^2\Big)^{\frac{1}{2}}\Big(\sum_{T\in {\cal T}_h} \|\bv\|_T^2\Big)^{\frac{1}{2}}\\&\quad+\Big(\sum_{T\in {\cal T}_h}h_T^{-1} \|\bQ_0\bu-\bu\|_{T} ^2+h_T \|\bQ_0\bu-\bu\|_{1,T} ^2\Big)^{\frac{1}{2}}\Big(\sum_{T\in {\cal T}_h}Ch_T^{-1}\|\bv\|_T^2\Big)^{\frac{1}{2}}\\
&\leq Ch^k\|\bu\|_{k+1}\Big(\sum_{T\in {\cal T}_h} \|\bv\|_T^2\Big)^{\frac{1}{2}},
\end{split}
\end{equation*}
 for any $\bv\in [P_r(T)]^3$. 
 
Letting $\bv=\nabla_w\times(\bu-\bQ_h\bu)$ gives 
$$
\sum_{T\in {\cal T}_h}(\nabla_w\times (\bu-\bQ_h\bu), \nabla_w\times (\bu-\bQ_h\bu))_T\leq 
 Ch^k\|\bu\|_{k+1}\3bar \bu-\bQ_h\bu \3bar_{V_h}.$$  
 
 This completes the proof of the lemma.
\end{proof}

\begin{theorem}\label{thm}
Let $k\geq 1$. Assume that the exact solutions of the Maxwell equations \eqref{model} possess sufficient regularity, such that  $\bu\in [H^{k+1}(\Omega)]^3$ and $p\in H^{k}(\Omega)$.  Let $(\bu_h, p_h)\in V_h\times W_h^0$ denote the  WG approximations obtained from the WG scheme \eqref{32}-\eqref{2}. 
Then, the following error estimate holds: 
 \begin{equation}\label{estimate1}
\3bar \be_h\3bar_{V_h}+\3bar \epsilon_h \3bar_{W_h} \leq Ch^k(\|\bu\|_{k+1}+\|p\|_k),
\end{equation} 
 \begin{equation}\label{estimate2}
(\sum_{T\in {\cal T}_h} (\epsilon_{h,0}, \epsilon_{h,0})_T)^{\frac{1}{2}}\leq Ch^{k}(\|\bu\|_{k+1}+\|p\|_k).
\end{equation} 
\end{theorem}
\begin{proof}
Letting $\bv=\bQ_h\bu-\bu_h=\bQ_h\bu-\bu+\be_h$ in (\ref{sehv}) and $q=\epsilon_h$ in (\ref{sehv2})  and adding the two equations, we have 
\begin{equation}\label{new5}
\begin{split}
& \3bar \be_h\3bar^2_{V_h} +\3bar \epsilon_h\3bar^2_{W_h}\\ =&\ell_1(\bu, \bQ_h\bu-\bu_h)+s(Q_hp, \epsilon_h)+\ell_2(\bu, \epsilon_h) \\& -a(\be_h, \bQ_h\bu-\bu)+b(\bQ_h\bu-\bu, \epsilon_h).\\
\end{split}
\end{equation} 
 Using \eqref{error2}, \eqref{normeq} and \eqref{errorestinew}, we have 
\begin{equation}\label{new6}
    \begin{split}
  &\ell_1(\bu, \bQ_h\bu-\bu_h)\\
  \leq & Ch^k\|\bu\|_{k+1} \3bar  \bQ_h\bu-\bu_h \3bar_{V_h}\\
   \leq & Ch^k\|\bu\|_{k+1} (\3bar  \bQ_h\bu-\bu \3bar_{V_h}+\3bar   \bu-\bu_h \3bar_{V_h})\\
  \leq &     Ch^{2k}\|\bu\|^2_{k+1}+ Ch^k\|\bu\|_{k+1}\3bar   \be_h \3bar_{V_h}.
\end{split}
\end{equation} 

Using  \eqref{error3},   \eqref{error4}, \eqref{normeq} and \eqref{errorestinew}, we have 
\begin{equation}\label{new2}
    \begin{split}
  & s(Q_hp, \epsilon_h)+\ell_2(\bu, \epsilon_h)    \\
  \leq &    Ch^k\|p\|_k \3bar\epsilon_h\3bar_{W_h}+ Ch^k\|\bu\|_{k+1} \3bar\epsilon_h\3bar_{W_h}.
 \end{split}
\end{equation}

Using Cauchy-Schwarz inequality and \eqref{errorestinew}, we have
\begin{equation}\label{new3}
    \begin{split}
a(\be_h, \bQ_h\bu-\bu)\leq \3bar \be_h\3bar_{V_h} \3bar \bQ_h\bu-\bu\3bar_{V_h} \leq Ch^k\|\bu\|_{k+1} \3bar \be_h\3bar.
 \end{split}
\end{equation}

Using \eqref{disgradient*} and the trace inequality \eqref{trace} gives
\begin{equation}\label{new1}
    \begin{split}
\sum_{T\in {\cal T}_h}\|\nabla_w \epsilon_h\|^2_T
\leq &\sum_{T\in {\cal T}_h} \|\nabla \epsilon_0\|^2_T+Ch_T^{-1} \|\epsilon_0-\epsilon_b\|^2_{\partial T}\\ 
\leq & \sum_{T\in {\cal T}_h} \|\nabla (Q_0p-p_0)\|^2_T+Ch_T^{-2}  \3bar\epsilon_h\3bar_{W_h}^2\\
\leq & Ch_T^{-2} (h_T^{2k}\|p\|^2_k+\3bar\epsilon_h\3bar^2_{W_h}).
\end{split}
\end{equation}
Using Cauchy-Schwarz inequality, Cauchy-Schwarz inequality  and \eqref{new1} gives 
\begin{equation}\label{new4}
    \begin{split}
        b(Q_h\bu-\bu, \epsilon_h)=&
        \sum_{T\in {\cal T}_h}(Q_0\bu-\bu, \nabla_w \epsilon_h)_T\\\leq & \big(\sum_{T\in {\cal T}_h} \|Q_0\bu-\bu\|_T^2\big)^{\frac{1}{2}}
   \big(\sum_{T\in {\cal T}_h} \|\nabla_w \epsilon_h\|_T^2\big)^{\frac{1}{2}}\\
   \leq & Ch^{k+1}\|\bu\|_{k+1} h_T^{-1} (h_T^{k}\|p\|_k+\3bar\epsilon_h\3bar_{W_h})\\
   \leq & Ch^{k}\|\bu\|_{k+1}  (h_T^{k}\|p\|_k+\3bar\epsilon_h\3bar_{W_h})\\
   \leq  & Ch^{2k}\|\bu\|^2_{k+1}  +C h_T^{2k}\|p\|^2_k  +Ch^{k}\|\bu\|_{k+1}  \3bar\epsilon_h\3bar_{W_h}. 
    \end{split}
\end{equation}

Substituting \eqref{new6}, \eqref{new2}, \eqref{new3}, \eqref{new4}
into \eqref{new5} gives \eqref{estimate1}.

Next, we prove \eqref{estimate2}. Let $\bw$ be the solution of $\nabla\cdot\bw=-\epsilon_0$. 
Letting $\bv_{\epsilon_h}=\bQ_h\bw=\{\bQ_0 \bw, \bQ_b\bw\}$  in \eqref{inf1} and  \eqref{sehv}  and using Cauchy-Schwartz inequality, \eqref{error2}, \eqref{normeq}, and \eqref{inf2}, gives 
\begin{equation*}
\begin{split}
&\quad   \sum_{T\in {\cal T}_h} (\epsilon_{h, 0}, \epsilon_{h, 0})_T \\&\leq |b(\bv_{\epsilon_h}, \epsilon_h)|+C|s(\epsilon_h, \epsilon_h)|\\
&\leq  |a(\be_h, \bv_{\epsilon_h})|+|\ell_1(\bu, \bv_{\epsilon_h})|+C\3bar \epsilon_h\3bar_{W_h}^2\\
&\leq  \3bar\be_h\3bar_{V_h} \3bar\bv_{\epsilon_h}\3bar_{V_h}+Ch^k\|\bu\|_{k+1}\|\bv_{\epsilon_h}\|_{1,h}+C\3bar \epsilon_h\3bar_{W_h}^2
\\
&\leq  \3bar\be_h\3bar_{V_h} \3bar\bv_{\epsilon_h}\3bar_{V_h}+Ch^k\|\bu\|_{k+1}\3bar\bv_{\epsilon_h}\3bar_{V_h}+C\3bar \epsilon_h\3bar_{W_h}^2\\
&\leq C( \3bar\be_h\3bar_{V_h}  +Ch^k\|\bu\|_{k+1})  (\sum_{T\in {\cal T}_h} ( \epsilon_{h, 0}, \epsilon_{h, 0})_T)^{\frac{1}{2}} +C\3bar \epsilon_h\3bar_{W_h}^2,\\
\end{split}
\end{equation*}
which, using \eqref{estimate1}, completes the proof of \eqref{estimate2}.

\end{proof}

\section{$L^2$ Error Estimates}
In this section, we derive  an $L^2$ Error Estimate for the WG approximation. To achieve this,   we consider an auxiliary problem of finding $(\bphi; \xi)$ such that
\begin{equation}\label{dual}
\begin{split}
\nabla \times (\nu\nabla\times \bphi)-\nabla \xi =&\bzeta_0, \qquad \text{in}\ \Omega,\\
\nabla\cdot\bphi =& 0, \qquad\text{in} \ \Omega,\\
\bphi\times\bn=& 0, \qquad\text{on} \ \partial\Omega,\\ 
\xi=& 0, \qquad\text{on} \ \partial\Omega. 
\end{split}
\end{equation} 
where we recall that the error function  $\be_h=\bu-\bu_h=\{\be_0,\be_b\}$. We further introduce $\bzeta_h=\bQ_h\bu-\bu_h=\{\bQ_0\bu-\bu_0,\bQ_b\bu-\bu_b\}=\{\bzeta_0,\bzeta_b\}$. 

Assume that the regularity property for the auxiliary problem \eqref{dual} holds, in the sense that $\bphi$ and $\xi$  satisfy the following regularity conditions:
 \begin{equation}\label{regu}
 \|\bphi\|_{1+\alpha}  +\|\xi\|_{\alpha}\leq C\| \bzeta_0\|,
\end{equation}
 where $0<\alpha\leq 1$.

 \begin{theorem}
Assume that the exact solutions of the Maxwell equations \eqref{model4} are sufficiently regular such that    $(\bu, p)\in [H^{k+1}(\Omega)]^3 \times H^k(\Omega)$. Let $(\bu_h, p_h)\in V_h\times W_h^0$ be the solutions of weak Galerkin scheme \eqref{32}-\eqref{2}. Additionally, assume that the regularity assumption \eqref{regu} for the dual problem \eqref{dual} holds true. Then, the following estimate holds:
 \begin{equation*}
      \|\be_0\|  \leq  Ch^{k+\alpha}(\|\bu\|_{1+k}+\|p\|_k).
 \end{equation*}
 \end{theorem}
    
 \begin{proof}
 Testing the first equation of \eqref{dual}  with $\bzeta_0$ yields:
 \begin{equation}\label{dq} 
  \|\bzeta_0\|^2= \sum_{T\in {\cal T}_h}(\nabla \times (\nu \nabla\times \bphi)-\nabla\xi, \bzeta_0)_T.
\end{equation} 

Letting $\bu=\bphi$, $p=\xi$ and $\bv=\bzeta_h$ in \eqref{eqq3} and \eqref{eqq2}, and letting  $\bv=\bQ_h \bphi$ in \eqref{sehv} and   $q=Q_h\xi$  in \eqref{sehv2}, and using \eqref{l-2}, the above equation can be formulated as follows
\begin{equation}\label{l2est1}
    \begin{split}
 &\|\bzeta_0\|^2\\=& \sum_{T\in {\cal T}_h}(\nu \nabla_w\times   \bphi, \nabla_w\times \bzeta_h)_T-\ell_1(\bphi, \bzeta_h)-\sum_{T\in {\cal T}_h}(\bzeta_0,  \bQ_0(\nabla\xi))_T \\ =& \sum_{T\in {\cal T}_h}(\nu \nabla_w\times   \bphi, \nabla_w\times (\bQ_h\bu-\bu))_T+(\nu \nabla_w\times   \bphi, \nabla_w\times \be_h)_T-\ell_1(\bphi, \bzeta_h)\\&-\sum_{T\in {\cal T}_h}(\bQ_0\bu-\bu, \nabla_w Q_h \xi)_T-\sum_{T\in {\cal T}_h}(\be_0,    \nabla_w Q_h \xi)_T  \\=& \sum_{T\in {\cal T}_h}(\nu \nabla_w\times   \bphi, \nabla_w\times (\bQ_h\bu-\bu))_T\\&+(\nu \nabla_w\times \bQ_h  \bphi, \nabla_w\times \be_h)_T+(\nu \nabla_w\times (\bphi-\bQ_h  \bphi), \nabla_w\times \be_h)_T\\& -\ell_1(\bphi, \bzeta_h) -\sum_{T\in {\cal T}_h}(\bQ_0\bu-\bu,  \nabla_w Q_h  \xi)_T  -\sum_{T\in {\cal T}_h}(\be_0,  \nabla _w Q_h\xi)_T 
 \\=& \sum_{T\in {\cal T}_h} (\nu \nabla_w\times   \bphi, \nabla_w\times (\bQ_h\bu-\bu))_T 
 +\ell_1(\bu, \bQ_h\bphi)+b(\bQ_h \bphi, \epsilon_h)\\&+\sum_{T\in {\cal T}_h}(\nu \nabla_w\times (\bphi-\bQ_h  \bphi), \nabla_w\times \be_h)_T  -\ell_1(\bphi, \bzeta_h) \\
 &-\sum_{T\in {\cal T}_h}(\bQ_0\bu-\bu,  \nabla_wQ_h  \xi)_T  + s(\epsilon_h, Q_h\xi)-s(Q_hp, Q_h\xi)-\ell_2(\bu, Q_h\xi) \\
 =&\sum_{i=1}^{9} I_i.
    \end{split}
\end{equation}

We will estimate $I_i$ for $i=1, \cdots, 9$ individually. For the term $I_1$, using \eqref{discurlcurl} gives
\begin{equation} \label{i0}
      \begin{split}
&  \sum_{T\in {\cal T}_h} (\bQ_0(\nu \nabla_w\times   \bphi), \nabla_w\times (\bQ_h\bu-\bu))_T\\
 = &\sum_{T\in {\cal T}_h}(\bQ_0\bu-\bu, \nabla\times (\bQ_0(\nu \nabla_w\times   \bphi)))_T
  -\langle (\bQ_b\bu- \bu)  \times \bn, \bQ_0(\nu \nabla_w\times   \bphi)\rangle_{\partial T}\\
  =&0,
  \end{split}
\end{equation}     
where we used the properties of the projection operators $\bQ_0$ and $\bQ_b$.
Using \eqref{i0}, Cauchy-Schwarz inequality,  \eqref{l}, and \eqref{errorestinew} gives
\begin{equation*}\label{i1}
      \begin{split}
|I_1|  =&| \sum_{T\in {\cal T}_h} (\nu \nabla_w\times   \bphi), \nabla_w\times (\bQ_h\bu-\bu))_T|\\=&| \sum_{T\in {\cal T}_h} ((I-\bQ_0)(\nu \nabla_w\times   \bphi), \nabla_w\times (\bQ_h\bu-\bu))_T| \\=&| \sum_{T\in {\cal T}_h} ((I-\bQ_0)(\nu {\cal Q}_h^r \nabla\times   \bphi), \nabla_w\times (\bQ_h\bu-\bu))_T| \\
\leq & \Big(\sum_{T\in {\cal T}_h}\|(I-\bQ_0)(\nu {\cal Q}_h^r \nabla\times   \bphi)\|_T^2\Big)^{\frac{1}{2}}\Big(\sum_{T\in {\cal T}_h}\|\nabla_w\times (\bQ_h\bu-\bu)\|_T^2\Big)^{\frac{1}{2}}\\
\leq  & Ch^{k+\alpha} \|\bphi\|_{1+\alpha}\|\bu\|_{k}.
  \end{split}
\end{equation*}

For the term $I_2$, applying the Cauchy-Schwarz inequality, the trace inequality \eqref{trace-inequality}, \eqref{est1} with $r=\alpha$ and $s=0, 1$,  \eqref{est2} with $r=k$ and $s=0, 1$, 
  we obtain
\begin{equation*}
    \begin{split}
   |I_2|=&| \ell_1(\bu, \bQ_h \bphi)|  \leq  
    |\sum_{T\in {\cal T} _h}\langle \nu (\bQ_b\bphi-\bQ_0\bphi)\times \bn, (I-{\cal Q}_h^{r})\nabla\times \bu\rangle_{\partial T}|   \\
    \leq & (\sum_{T\in {\cal T} _h}\| \nu (\bQ_b\bphi-\bQ_0\bphi)\times \bn\|^2_{\partial T})^\frac{1}{2} (\sum_{T\in {\cal T} _h}\| (I-{\cal Q}_h^{r})\nabla\times \bu\|^2_{\partial T})^\frac{1}{2}\\
   \leq & (\sum_{T\in {\cal T} _h}h_T^{-1}\| \bphi-\bQ_0\bphi \|^2_{T}+h_T\| \bphi-\bQ_0\bphi \|^2_{1, T})^\frac{1}{2} \\&\cdot (\sum_{T\in {\cal T} _h}h_T^{-1}\| (I-{\cal Q}_h^{r})\nabla\times \bu\|^2_{T}+h_T \| (I-{\cal Q}_h^{r})\nabla\times \bu\|^2_{1, T})^\frac{1}{2}\\   
  \leq & Ch^{k+\alpha} \|\bu\|_{k+1}\|\bphi\|_{1+\alpha}.
    \end{split}
\end{equation*}

Regarding to $I_3$,  it follows from \eqref{disgradient*}, the usual integration by parts, the second equation in \eqref{dual} that 
  \begin{equation*}\label{l2est2}
      \begin{split}
  I_3= b(\bQ_h\bphi, \epsilon_h)=&\sum_{T\in {\cal T}_h}(\bQ_0\bphi, \nabla_w\epsilon_h)_T
   \\
  =&\sum_{T\in {\cal T}_h}  (\nabla \epsilon_0, \bQ_0\bphi)_T-\langle \epsilon_0-\epsilon_b, \bQ_0\bphi\cdot\bn\rangle_{\partial T}\\
 =&\sum_{T\in {\cal T}_h}  (\nabla \epsilon_0,  \bphi)_T-\langle \epsilon_0-\epsilon_b, \bQ_0\bphi\cdot\bn\rangle_{\partial T}\\
 =&\sum_{T\in {\cal T}_h}  -( \epsilon_0,  \nabla \cdot\bphi)_T+\langle \epsilon_0, \bphi\cdot\bn\rangle_{\partial T}-\langle \epsilon_0-\epsilon_b, \bQ_0\bphi\cdot\bn\rangle_{\partial T}\\ 
  =&\sum_{T\in {\cal T}_h}    \langle \epsilon_0-\epsilon_b, \bphi\cdot\bn\rangle_{\partial T}-\langle \epsilon_0-\epsilon_b, \bQ_0\bphi\cdot\bn\rangle_{\partial T}\\ 
  =&\ell_2(\bphi, \epsilon_h),
  \end{split}
  \end{equation*}
where we used $\sum_{T\in {\cal T}_h}\langle\epsilon_b, \bphi\cdot\bn\rangle_{\partial T} =\langle \epsilon_b, \bphi\cdot\bn\rangle_{\partial\Omega}=0$ due to $\epsilon_b=0$ on $\partial\Omega$.
Using the Cauchy-Schwarz inequality, the trace inequality \eqref{trace-inequality}, \eqref{est1} with $r=\alpha$ and $s=0, 1$,  the error estimate \eqref{estimate1},  we obtain
 \begin{equation*}
     \begin{split}
 |I_3|=&|\ell_2(\bphi, \epsilon_h)| = |  \sum_{T\in{\cal T}_h} \langle \epsilon_0-\epsilon_b, (I-\bQ_0)\bphi \cdot\bn\rangle_{\pT} |\\
 \leq &(\sum_{T\in {\cal T} _h}h_T\|\epsilon_0-\epsilon_b\|^2_{\partial T})^\frac{1}{2} (\sum_{T\in {\cal T} _h}h_T^{-1}\|(I-\bQ_0)\bphi \cdot\bn\|^2_{\partial T})^\frac{1}{2}\\
 \leq & \3bar \epsilon_h\3bar_{W_h} (\sum_{T\in {\cal T} _h}h_T^{-2}\|(I-\bQ_0)\bphi \cdot\bn\|^2_{T}+\|(I-\bQ_0)\bphi \cdot\bn\|^2_{1, T})^\frac{1}{2}\\
 \leq & Ch^{k+\alpha}(\|\bu\|_{k+1}+\|p\|_k) \|\bphi\|_{1+\alpha}.
     \end{split}
 \end{equation*}

For the term $I_4$, using Cauchy-Schwarz inequality, \eqref{errorestinew} with $k=\alpha$, \eqref{estimate1}, we have
 \begin{equation*}
     \begin{split}
    |I_4|= &|\sum_{T\in {\cal T} _h} (\nu \nabla_w\times (\bphi-\bQ_h  \bphi), \nabla_w\times \be_h)_T|\\
   \leq & (\sum_{T\in {\cal T} _h} \|\nu \nabla_w\times (\bphi-\bQ_h  \bphi)\|_T^2)^{\frac{1}{2}}  (\sum_{T\in {\cal T} _h} \| \nabla_w\times \be_h\|_T|^2)^{\frac{1}{2}} \\
   \leq  &  \3bar \bphi-\bQ_h  \bphi\3bar_{V_h}  \3bar \be_h  \3bar_{V_h}
\\
\leq &Ch^{k+\alpha}\|\bphi\|_{1+\alpha} (\|\bu\|_{k+1}+\|p\|_k).
     \end{split}
 \end{equation*}

For the term $I_5$, using the Cauchy-Schwarz inequality, \eqref{normeq}, the trace inequality \eqref{trace-inequality}, \eqref{est2} with $r=\alpha$ and $s=0, 1$,  the estimate \eqref{estimate1}, we have
\begin{equation*}
    \begin{split}
 |I_5|=&|\ell_1(\bphi, \bzeta_h)|=  \sum_{T\in {\cal T} _h}\langle \nu (\be_b-\be_0)\times \bn, (I-{\cal Q}_h^{r})\nabla\times \bphi\rangle_{\partial T}\\
 \leq  & (\sum_{T\in {\cal T} _h}h_T^{-1}\|  \nu (\be_b-\be_0)\times \bn\|^2_{\partial T})^\frac{1}{2} (\sum_{T\in {\cal T} _h} h_T\|(I-{\cal Q}_h^{r})\nabla\times \bphi\|^2_{\partial T})^\frac{1}{2} \\ 
   \leq & \|\be_h\|_{1,h} (\sum_{T\in {\cal T} _h} \|(I-{\cal Q}_h^{r})\nabla\times \bphi\|^2_{T}+h_T^2\|(I-{\cal Q}_h^{r})\nabla\times \bphi\|^2_{T})^\frac{1}{2} \\
   \leq & C\3bar \be_h\3bar_{V_h} h^\alpha\|\bphi\|_{1+\alpha} 
   \leq Ch^{k+\alpha}(\|\bu\|_{1+k}+\|p\|_k)\|\bphi\|_{1+\alpha}.
   \end{split}
\end{equation*}

For the term $I_6$, using the Cauchy-Schwarz inequality,  \eqref{l-2},  we get
\begin{equation*}
    \begin{split}
|I_6|=&|\sum_{T\in {\cal T}_h}(\bQ_0\bu-\bu,  \nabla_w Q_h  \xi)_T|
\\=& |\sum_{T\in {\cal T}_h}(\bQ_0\bu-\bu,   \bQ_0\nabla    \xi)_T|\\
=&0.
 \end{split}
\end{equation*} 

 For the term $I_7$, using the Cauchy-Schwarz inequality, the trace inequality \eqref{trace-inequality}, \eqref{est4} with $r=\alpha$ and $s=0, 1$,  the error estimate \eqref{estimate1}, we have
\begin{equation*}
    \begin{split}
   |I_7|=&|s(\epsilon_h, Q_h\xi)|=|\sum_{T\in {\cal T}_h} h_T\langle \epsilon_0-\epsilon_b, Q_0\xi-Q_b\xi\rangle_{\partial T}|\\
   \leq & (\sum_{T\in {\cal T} _h}h_T\|\epsilon_0-\epsilon_b\|^2_{\partial T})^\frac{1}{2} (\sum_{T\in {\cal T} _h}h_T \|Q_0\xi-Q_b\xi\|^2_{\partial T})^\frac{1}{2}\\
   \leq & C\3bar \epsilon_h\3bar_{W_h}(\sum_{T\in {\cal T} _h} \|Q_0\xi- \xi\|^2_{T}+h_T^{2}\|Q_0\xi- \xi\|^2_{1, T})^\frac{1}{2}\\
   \leq & Ch^{k+\alpha}(\|\bu\|_{k+1}+\|p\|_k) \|\xi\|_{\alpha}.
    \end{split}
\end{equation*}
 
 For the term $I_8$, using the Cauchy-Schwarz inequality, the trace inequality \eqref{trace-inequality}, \eqref{est4} with $r=\alpha$ and $r=k$ and $s=0, 1$,  we have
 \begin{equation*}
          \begin{split}
             |I_8|=&|s(Q_h p, Q_h\xi)|=
        |\sum_{T\in {\cal T}_h} h_T\langle Q_0p-Q_bp, Q_0\xi-Q_b\xi\rangle_{\partial T}|\\
     \leq  & (\sum_{T\in {\cal T} _h}h_T\|Q_0p-Q_bp\|^2_{\partial T})^\frac{1}{2} (\sum_{T\in {\cal T} _h}h_T \|Q_0\xi-Q_b\xi\|^2_{\partial T})^\frac{1}{2}
     \\ \leq&  (\sum_{T\in {\cal T} _h} \|Q_0p-p\|^2_{ T}+h_T^2\|Q_0p- p\|^2_{1, T})^\frac{1}{2} \\&\cdot(\sum_{T\in {\cal T} _h} \|Q_0\xi-\xi\|^2_{ T}+h_T^2 \|Q_0\xi-\xi\|^2_{1, T})^\frac{1}{2}\\
  \leq&Ch^{k+\alpha} \|p\|_k \|\xi\|_{\alpha}.
     \end{split}
      \end{equation*}

  For the term $I_9$, using  the Cauchy-Schwarz inequality, the trace inequality \eqref{trace-inequality}, \eqref{est1} with $r=k$ and $s=0, 1$,  \eqref{est4} with $r=\alpha$ and $s=0, 1$,  we have
  \begin{equation*}
      \begin{split}
     |I_9|=& |\ell_2(\bu, Q_h\xi)| =|\sum_{T\in{\cal T}_h} \langle Q_0\xi-Q_b\xi, (I-\bQ_0)\bu \cdot\bn\rangle_{\pT}|\\    
   \leq  & (\sum_{T\in {\cal T} _h}\|Q_0\xi-Q_b\xi\|^2_{\partial T})^\frac{1}{2} (\sum_{T\in {\cal T} _h} \|(I-\bQ_0)\bu \cdot\bn\|^2_{\partial T})^\frac{1}{2} \\
   \leq  & (\sum_{T\in {\cal T} _h}h_T^{-1}\|Q_0\xi-\xi\|^2_{ T}+h_T\|Q_0\xi-\xi\|^2_{1, T})^\frac{1}{2} \\&\cdot(\sum_{T\in {\cal T} _h} h_T^{-1}\|(I-\bQ_0)\bu \cdot\bn\|^2_{T}+h_T\|(I-\bQ_0)\bu \cdot\bn\|^2_{1, T})^\frac{1}{2}\\
   \leq &  Ch^{k+\alpha} \|\bu\|_{k+1} \|\xi\|_{\alpha}.  
   \end{split}
  \end{equation*}

    Substituting the estimates of $I_i$ for $i=1, \cdots, 9$ into \eqref{l2est1} and using the regularity assumption \eqref{regu} gives
    $$
    \|\bzeta_0\|^2\leq Ch^{k+\alpha}(\|\bu\|_{1+k}+\|p\|_k)\|\bzeta_0\|.
    $$
This, together with the triangle inequality, we have
      \begin{equation*}
      \begin{split}
            \|\be_0\|\leq & \|\bzeta_0\|+\|\bu-\bQ_0\bu\|\\\leq &   Ch^{k+\alpha}(\|\bu\|_{1+k}+\|p\|_k)+Ch^{k+1}\|\bu\|_{k+1}\\\leq &  Ch^{k+\alpha}(\|\bu\|_{1+k}+\|p\|_k).
      \end{split}
      \end{equation*}

   This completes the proof of the theorem.
 \end{proof}

\section{Numerical test}
In the first test,  we solve \eqref{weakform} with $\nu=1$ on the unit square domain $\Omega
  =(0,1)\times(0,1)$.
We choose $\b f$ and $g$ such that the exact solution is
\an{\label{e1} \b u=\p{ 4(x^2-2x^3+x^4)(2y-6y^2+4y^3) \\ -4(y^2-2y^3+y^4)(2x-6x^2+4x^3)}, \
                p=(x-x^2)(y-y^2).  }

     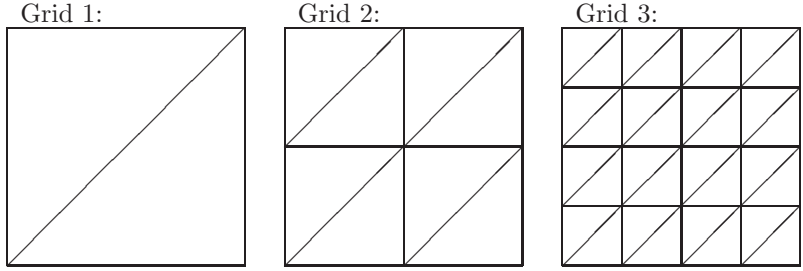
\begin{figure}[H] \setlength\unitlength{1pt}\begin{center}
    \begin{picture}(300,100)(0,0)
     \def\mc{\begin{picture}(90,90)(0,0)
       \put(0,0){\line(1,0){90}} \put(0,90){\line(1,0){90}}
      \put(0,0){\line(0,1){90}}  \put(90,0){\line(0,1){90}}  \put(0,0){\line(1,1){90}}  
      \end{picture}}

    \put(0,0){\mc} \put(5,92){Grid 1:} \put(110,92){Grid 2:}\put(215,92){Grid 3:}
      \put(105,0){\setlength\unitlength{0.5pt}\begin{picture}(90,90)(0,0)
    \put(0,0){\mc}\put(90,0){\mc}\put(0,90){\mc}\put(90,90){\mc}\end{picture}}
      \put(210,0){\setlength\unitlength{0.25pt}\begin{picture}(90,90)(0,0)
    \multiput(0,0)(90,0){4}{\multiput(0,0)(0,90){4}{\mc}} \end{picture}}
    \end{picture}
 \caption{The first three grids for the computation in
    Tables \ref{t-1}--\ref{t-2}. }\label{grid1} 
    \end{center} \end{figure}

We compute the finite element solutions for \eqref{e1} on uniform 
       triangular grids shown in Figure \ref{grid1} by 
  the  $P_k$/$P_{k-1}$ WG finite elements, defined in \eqref{Wk}--\eqref{Vk}, 
      for $k=1,2,3,4$ and $5$.
The results are listed in Tables \ref{t-1}--\ref{t-2}. 
The optimal order of convergence is achieved for all solutions in all norms.

\begin{table}[H]
  \centering  \renewcommand{\arraystretch}{1.1}
  \caption{The error and the computed order of convergence  ($k=1,2,3$)
     for the solution \eqref{e1} on Figure \ref{grid1} meshes. }
  \label{t-1}
\begin{tabular}{c|cc|cc|cc}
\hline
Grid &  $\|\b u-\b u_0\|_0$  & $O(h^r)$ & $\3bar\b u-\b u_h\3bar$  & $O(h^r)$ &
      $\|p -p_0\|_0$  & $O(h^r)$    \\
\hline&\multicolumn{6}{c}{ By the $P_1$/$P_{0}$ WG finite element.}\\
\hline 
 6&    0.121E-03 &  2.0&    0.106E-01 &  1.0&    0.738E-02 &  1.0 \\
 7&    0.302E-04 &  2.0&    0.532E-02 &  1.0&    0.372E-02 &  1.0 \\
 8&    0.754E-05 &  2.0&    0.266E-02 &  1.0&    0.186E-02 &  1.0 \\
\hline&\multicolumn{6}{c}{ By the $P_2$/$P_{1}$ WG finite element.}\\
\hline 
 5&    0.173E-04 &  3.0&    0.105E-02 &  2.0&    0.381E-03 &  2.0 \\
 6&    0.219E-05 &  3.0&    0.263E-03 &  2.0&    0.948E-04 &  2.0 \\
 7&    0.275E-06 &  3.0&    0.658E-04 &  2.0&    0.236E-04 &  2.0 \\ 
\hline&\multicolumn{6}{c}{ By the $P_3$/$P_{2}$ WG finite element.}\\
\hline 
 4&    0.920E-05 &  3.9&    0.575E-03 &  2.9&    0.934E-04 &  2.9 \\
 5&    0.596E-06 &  3.9&    0.733E-04 &  3.0&    0.122E-04 &  2.9 \\
 6&    0.378E-07 &  4.0&    0.921E-05 &  3.0&    0.156E-05 &  3.0 \\
\hline
    \end{tabular}%
\end{table}%

\begin{table}[H]
  \centering  \renewcommand{\arraystretch}{1.1}
  \caption{The error and the computed order of convergence ($k=4,5$)
     for the solution \eqref{e1} on Figure \ref{grid1} meshes. }
  \label{t-2}
\begin{tabular}{c|cc|cc|cc}
\hline
Grid &  $\|\b u-\b u_0\|_0$  & $O(h^r)$ & $\3bar\b u-\b u_h\3bar$  & $O(h^r)$ &
      $\|p -p_0\|_0$  & $O(h^r)$    \\
\hline&\multicolumn{6}{c}{ By the $P_4$/$P_{3}$ WG finite element.}\\
\hline 
 4&    0.682E-06 &  4.9&    0.689E-04 &  3.9&    0.428E-05 &  4.0 \\
 5&    0.217E-07 &  5.0&    0.437E-05 &  4.0&    0.268E-06 &  4.0 \\
 6&    0.923E-09 &  4.6&    0.556E-06 &  3.0&    0.168E-07 &  4.0 \\
\hline&\multicolumn{6}{c}{ By the $P_5$/$P_{4}$ WG finite element.}\\
\hline 
 3&    0.208E-05 &  5.8&    0.158E-03 &  4.9&    0.413E-05 &  4.6 \\
 4&    0.336E-07 &  6.0&    0.504E-05 &  5.0&    0.139E-06 &  4.9 \\
 5&    0.990E-09 &  5.1&    0.443E-06 &  3.5&    0.568E-08 &  4.6 \\
\hline
    \end{tabular}%
\end{table}%

We next compute the finite element solutions for \eqref{e1} on non-convex polygonal 
        grids shown in Figure \ref{grid2} by 
  the  $P_k$/$P_{k-1}$ WG finite elements, defined in \eqref{Wk}--\eqref{Vk}, 
      for $k=1,2,3$ and $4$.
The results are listed in Tables \ref{t-3}--\ref{t-4}. 
The optimal order of convergence is achieved for all solutions in all norms.

     \begin{figure}[H] \setlength\unitlength{1.2pt}\begin{center}
    \begin{picture}(280,98)(0,0)
     \def\mc{\begin{picture}(90,90)(0,0)
       \put(0,0){\line(1,0){90}} \put(0,90){\line(1,0){90}}
      \put(0,0){\line(0,1){90}}  \put(90,0){\line(0,1){90}} 
      
       \put(15,60){\line(2,-1){60}}  
       \put(0,0){\line(5,2){75}}         \put(90,90){\line(-5,-2){75}}  
      \end{picture}}

    \put(0,0){\mc} \put(0,92){Grid 1:} \put(95,92){Grid 2:}  \put(190,92){Grid 3:}
    
      \put(95,0){\setlength\unitlength{0.6pt}\begin{picture}(90,90)(0,0)
    \put(0,0){\mc}\put(90,0){\mc}\put(0,90){\mc}\put(90,90){\mc}\end{picture}}
     \put(190,0){\setlength\unitlength{0.3pt}\begin{picture}(90,90)(0,0)
      \multiput(0,0)(90,0){4}{\multiput(0,0)(0,90){4}{\mc}} \end{picture}}
    \end{picture}
 \caption{The first three non-convex polygonal grids for the computation in
    Tables \ref{t-3}--\ref{t-4}. }\label{grid2} 
    \end{center} \end{figure}
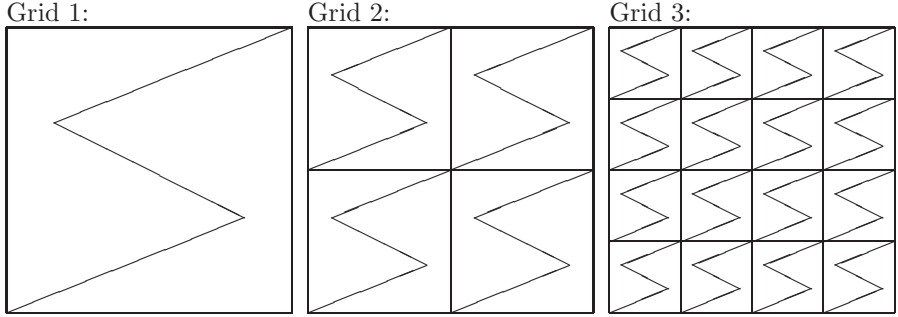

\begin{table}[H]
  \centering  \renewcommand{\arraystretch}{1.1}
  \caption{The error and the computed order of convergence ($k=1,2$)
     for the solution \eqref{e1} on Figure \ref{grid2} meshes. }
  \label{t-3}
\begin{tabular}{c|cc|cc|cc}
\hline
Grid &  $\|\b u-\b u_0\|_0$  & $O(h^r)$ & $\3bar\b u-\b u_h\3bar$  & $O(h^r)$ &
      $\|p -p_0\|_0$  & $O(h^r)$    \\
\hline&\multicolumn{6}{c}{ By the $P_1$/$P_{0}$ WG finite element.}\\
\hline 
 5&    0.174E-02 &  4.3&    0.168E+00 &  2.2&    0.707E-01 &  3.3 \\
 6&    0.102E-02 &  0.8&    0.111E+00 &  0.6&    0.809E-01 &  0.0 \\
 7&    0.398E-03 &  1.4&    0.743E-01 &  0.6&    0.637E-01 &  0.3 \\
 8&    0.282E-04 &  3.8&    0.211E-01 &  1.8&    0.872E-02 &  2.9 \\
\hline&\multicolumn{6}{c}{ By the $P_2$/$P_{1}$ WG finite element.}\\
\hline 
 5&    0.672E-04 &  3.0&    0.203E-01 &  2.0&    0.135E-02 &  2.0 \\
 6&    0.845E-05 &  3.0&    0.509E-02 &  2.0&    0.334E-03 &  2.0 \\
 7&    0.106E-05 &  3.0&    0.127E-02 &  2.0&    0.828E-04 &  2.0 \\
\hline
    \end{tabular}%
\end{table}%

\begin{table}[H]
  \centering  \renewcommand{\arraystretch}{1.1}
  \caption{The error and the computed order of convergence ($k=3,4$)
     for the solution \eqref{e1} on Figure \ref{grid2} meshes. }
  \label{t-4}
\begin{tabular}{c|cc|cc|cc}
\hline
Grid &  $\|\b u-\b u_0\|_0$  & $O(h^r)$ & $\3bar\b u-\b u_h\3bar$  & $O(h^r)$ &
      $\|p -p_0\|_0$  & $O(h^r)$    \\
\hline&\multicolumn{6}{c}{ By the $P_3$/$P_{2}$ WG finite element.}\\
\hline 
 3&    0.167E-02 &  4.3&    0.142E+00 &  2.4&    0.137E-01 &  3.8 \\
 4&    0.200E-03 &  3.1&    0.200E-01 &  2.8&    0.435E-02 &  1.7 \\
 5&    0.136E-04 &  3.9&    0.257E-02 &  3.0&    0.598E-03 &  2.9 \\
 6&    0.634E-06 &  4.4&    0.319E-03 &  3.0&    0.512E-04 &  3.5 \\
\hline&\multicolumn{6}{c}{ By the $P_4$/$P_{3}$ WG finite element.}\\
\hline 
 3&    0.524E-03 &  4.6&    0.642E-01 &  3.7&    0.423E-02 &  3.3 \\
 4&    0.176E-04 &  4.9&    0.420E-02 &  3.9&    0.308E-03 &  3.8 \\
 5&    0.562E-06 &  5.0&    0.267E-03 &  4.0&    0.201E-04 &  3.9 \\
\hline
    \end{tabular}%
\end{table}%

We next compute the finite element solutions for \eqref{e1} on non-convex polygonal 
        grids shown in Figure \ref{grid3} by 
  the  $P_k$/$P_{k-1}$ WG finite elements, defined in \eqref{Wk}--\eqref{Vk}, 
      for $k=1,2$ and $3$.
The results are listed in Table  \ref{t-5}. 
The optimal order of convergence is achieved for all solutions in all norms.

     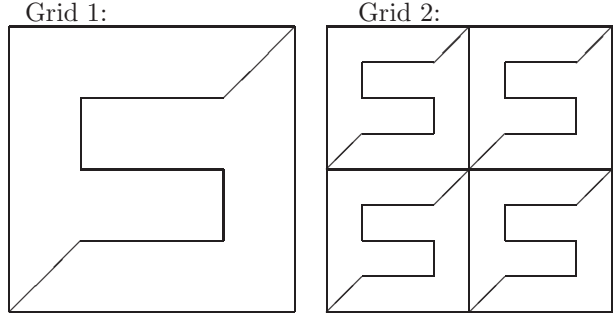
\begin{figure}[H] \setlength\unitlength{1.2pt}\begin{center}
    \begin{picture}(190,99)(0,0)
     \def\mc{\begin{picture}(90,90)(0,0)
       \put(0,0){\line(1,0){90}} \put(0,90){\line(1,0){90}}
      \put(0,0){\line(0,1){90}}  \put(90,0){\line(0,1){90}} 
      
       \put(22.5,22.5){\line(1,0){45}}  \put(67.5,22.5){\line(0,1){22.5}}  
       \put(22.5,45){\line(1,0){45}}  \put(22.5,45){\line(0,1){22.5}}  
       \put(22.5,67.5){\line(1,0){45}}  
       \put(0,0){\line(1,1){22.5}}         \put(90,90){\line(-1,-1){22.5}}  
      \end{picture}}

    \put(0,0){\mc} \put(5,92){Grid 1:} \put(110,92){Grid 2:} 
      \put(100,0){\setlength\unitlength{0.6pt}\begin{picture}(90,90)(0,0)
    \put(0,0){\mc}\put(90,0){\mc}\put(0,90){\mc}\put(90,90){\mc}\end{picture}}
    \end{picture}
 \caption{The first two grids for the computation in
    Table \ref{t-5}. }\label{grid3} 
    \end{center} \end{figure}

\begin{table}[H]
  \centering  \renewcommand{\arraystretch}{1.1}
  \caption{The error and the computed order of convergence ($k=1,2,3$)
     for the solution \eqref{e1} on Figure \ref{grid3} meshes. }
  \label{t-5}
\begin{tabular}{c|cc|cc|cc}
\hline
Grid &  $\|\b u-\b u_0\|_0$  & $O(h^r)$ & $\3bar\b u-\b u_h\3bar$  & $O(h^r)$ &
      $\|p -p_0\|_0$  & $O(h^r)$    \\
\hline&\multicolumn{6}{c}{ By the $P_1$/$P_{0}$ WG finite element.}\\
\hline 
 6&    0.265E-03 &  2.0&    0.798E-01 &  1.0&    0.124E-01 &  1.0 \\
 7&    0.662E-04 &  2.0&    0.400E-01 &  1.0&    0.620E-02 &  1.0 \\
 8&    0.166E-04 &  2.0&    0.200E-01 &  1.0&    0.311E-02 &  1.0 \\
\hline&\multicolumn{6}{c}{ By the $P_2$/$P_{1}$ WG finite element.}\\
\hline 
 5&    0.168E-03 &  3.4&    0.183E-01 &  2.0&    0.562E-02 &  2.4 \\
 6&    0.373E-04 &  2.2&    0.507E-02 &  1.8&    0.258E-02 &  1.1 \\
 7&    0.608E-05 &  2.6&    0.138E-02 &  1.9&    0.850E-03 &  1.6 \\ 
\hline&\multicolumn{6}{c}{ By the $P_3$/$P_{2}$ WG finite element.}\\
\hline 
 4&    0.115E-03 &  3.8&    0.144E-01 &  2.8&    0.135E-02 &  2.7 \\
 5&    0.751E-05 &  3.9&    0.185E-02 &  3.0&    0.187E-03 &  2.8 \\
 6&    0.581E-06 &  3.7&    0.275E-03 &  2.8&    0.237E-04 &  3.0 \\
\hline
    \end{tabular}%
\end{table}%

In the 3D numerical computation,  the domain for problem \eqref{weakform}
   is the unit cube $\Omega=(0,1)\times(0,1)\times(0,1)$.
We choose an $\b f$ and a $g$ in \eqref{weakform} so that the exact solution is
\an{\label{e3} \ad{  
  \b u &=\p{-2^{10}(x-1)^2x^2(y-1)^2y^2(z  -3z^2 +2z^3 ) \\
           \ \; 2^{10}(x-1)^2x^2(y-1)^2y^2(z  -3z^2 +2z^3 )\\
            2^{10} [(x -3x^2 +2x^3 )(y^2-y)^2-(x^2-x)^2(y-3y^2 +2y^3 )](z^2-z)^2},\\
            p&= (x-x^2)(y-y^2)(z-z^2). }  }

We   compute the finite element solutions for \eqref{e3} on the
        grids shown in Figure \ref{grid4} by 
  the  $P_k$/$P_{k-1}$ WG finite elements, defined in \eqref{Wk}--\eqref{Vk}, 
      for $k=1,2$ and $3$.
The results are listed in Table  \ref{t-6}. 
The optimal order of convergence is achieved for all solutions in all norms.

\begin{figure}[H] 
\begin{center}
 \setlength\unitlength{1pt}
    \begin{picture}(320,118)(0,3)
    \put(0,0){\begin{picture}(110,110)(0,0) \put(25,102){Grid 1:}
       \multiput(0,0)(80,0){2}{\line(0,1){80}}  \multiput(0,0)(0,80){2}{\line(1,0){80}}
       \multiput(0,80)(80,0){2}{\line(1,1){20}} \multiput(0,80)(20,20){2}{\line(1,0){80}}
       \multiput(80,0)(0,80){2}{\line(1,1){20}}  \multiput(80,0)(20,20){2}{\line(0,1){80}}
    \put(80,0){\line(-1,1){80}}\put(80,0){\line(1,5){20}}\put(80,80){\line(-3,1){60}}
      \end{picture}}
    \put(110,0){\begin{picture}(110,110)(0,0)\put(25,102){Grid 2:}
       \multiput(0,0)(40,0){3}{\line(0,1){80}}  \multiput(0,0)(0,40){3}{\line(1,0){80}}
       \multiput(0,80)(40,0){3}{\line(1,1){20}} \multiput(0,80)(10,10){3}{\line(1,0){80}}
       \multiput(80,0)(0,40){3}{\line(1,1){20}}  \multiput(80,0)(10,10){3}{\line(0,1){80}}
    \put(80,0){\line(-1,1){80}}\put(80,0){\line(1,5){20}}\put(80,80){\line(-3,1){60}}
       \multiput(40,0)(40,40){2}{\line(-1,1){40}} 
        \multiput(80,40)(10,-30){2}{\line(1,5){10}}
        \multiput(40,80)(50,10){2}{\line(-3,1){30}}
      \end{picture}}
    \put(220,0){\begin{picture}(110,110)(0,0) \put(25,102){Grid 3:}
       \multiput(0,0)(20,0){5}{\line(0,1){80}}  \multiput(0,0)(0,20){5}{\line(1,0){80}}
       \multiput(0,80)(20,0){5}{\line(1,1){20}} \multiput(0,80)(5,5){5}{\line(1,0){80}}
       \multiput(80,0)(0,20){5}{\line(1,1){20}}  \multiput(80,0)(5,5){5}{\line(0,1){80}}
    \put(80,0){\line(-1,1){80}}\put(80,0){\line(1,5){20}}\put(80,80){\line(-3,1){60}}
       \multiput(40,0)(40,40){2}{\line(-1,1){40}} 
        \multiput(80,40)(10,-30){2}{\line(1,5){10}}
        \multiput(40,80)(50,10){2}{\line(-3,1){30}}

       \multiput(20,0)(60,60){2}{\line(-1,1){20}}   \multiput(60,0)(20,20){2}{\line(-1,1){60}} 
        \multiput(80,60)(15,-45){2}{\line(1,5){5}} \multiput(80,20)(5,-15){2}{\line(1,5){15}}
        \multiput(20,80)(75,15){2}{\line(-3,1){15}}\multiput(60,80)(25,5){2}{\line(-3,1){45}}
      \end{picture}}

    \end{picture} 
    \end{center} 
\caption{ The first three grids for the computation 
    in Table  \ref{t-6}.  } 
\label{grid4}
\end{figure}
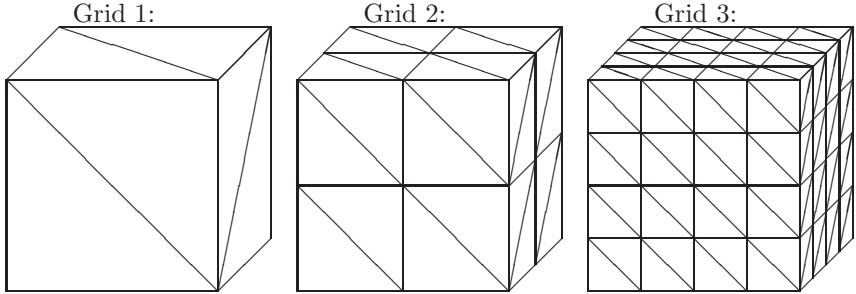

\begin{table}[H]
  \centering  \renewcommand{\arraystretch}{1.1}
  \caption{The error and the computed order of convergence ($k=1,2,3$)
     for the 3D solution \eqref{e3} on Figure \ref{grid4} meshes. }
  \label{t-6}
\begin{tabular}{c|cc|cc|cc}
\hline
Grid &  $\|\b u-\b u_0\|_0$  & $O(h^r)$ & $\3bar\b u-\b u_h\3bar$  & $O(h^r)$ &
      $\|p -p_0\|_0$  & $O(h^r)$    \\
\hline&\multicolumn{6}{c}{ By the $P_1$/$P_{0}$ WG finite element.}\\
\hline 
 3 &   0.273E-01 &1.95 &   0.627E+00 &0.65 &   0.117E-01 &2.40 \\
 4 &   0.734E-02 &1.89 &   0.341E+00 &0.88 &   0.320E-02 &1.87 \\
 5 &   0.212E-02 &1.80 &   0.176E+00 &0.96 &   0.828E-03 &1.95 \\
\hline&\multicolumn{6}{c}{ By the $P_2$/$P_{1}$ WG finite element.}\\
\hline 
 3 &   0.562E-02 &3.04 &   0.157E+00 &1.75 &   0.230E-02 &3.04 \\
 4 &   0.697E-03 &3.01 &   0.427E-01 &1.88 &   0.247E-03 &3.22 \\
 5 &   0.983E-04 &2.83 &   0.109E-01 &1.97 &   0.278E-04 &3.15 \\
\hline&\multicolumn{6}{c}{ By the $P_3$/$P_{2}$ WG finite element.}\\
\hline 
 2 &   0.160E-01 &2.98 &   0.207E+00 &1.96 &   0.656E-02 &3.05 \\
 3 &   0.952E-03 &4.07 &   0.333E-01 &2.63 &   0.440E-03 &3.90 \\
 4 &   0.593E-04 &4.01 &   0.449E-02 &2.89 &   0.326E-04 &3.75 \\
\hline
    \end{tabular}%
\end{table}%

\bibliographystyle{abbrv}
\bibliography{Ref}

\end{document}